\documentclass[12pt]{amsart}
\usepackage{color}
\usepackage{amssymb}
\usepackage{amsthm}
\usepackage{amsmath}
\usepackage[all]{xy}
\usepackage{graphicx}
\usepackage{enumerate}
\usepackage{bussproofs}
\usepackage{mathrsfs}
\usepackage{array}

\newcommand{\algp}{\mathbf P}

\newcommand{\alga}{\mathbf A}

\newcommand{\pp}{\perp}
\newcommand{\N}{\mathcal{N}}
\newcommand{\ra}{\Rightarrow}
\newcommand{\nlra}{\not\Leftrightarrow}
\newcommand{\tr}{\mathsf{t}}
\newcommand{\fa}{\mathsf{f}}

\newcommand{\nl}{\mathsf{NL}}
\newcommand{\nlas}{{\mathsf{NL}^{\mathrm{as}}}}

\newcommand{\nel}{\mathsf{NeL}}
\newcommand{\nelas}{{\mathsf{NeL}^{\mathrm{as}}}}
\newcommand{\nelw}{{\mathsf{NeL}_{w}}}

\newcommand{\Eqts}{$\mathrm{(Eq)}_{1}^{s}$}

\newcommand{\Mps}{$\mathrm{(MP)}^{s}$}
\newcommand{\Ces}{$\mathrm{(CE)}^{s}$}

\newtheorem{theorem}{Theorem}[section]
\theoremstyle{definition}
\newtheorem{definition}{Definition}[section]

\newtheorem{proposition}[theorem]{Proposition}
\theoremstyle{definition}
\newtheorem{remark}{Remark}
\newtheorem{question}[theorem]{Problem}
\newtheorem{example}[theorem]{Example}
\newtheorem{corollary}[theorem]{Corollary}
\title{Considerations on Everett~J.~Nelson's connexive logic
}
\author{Davide Fazio}
\author{Raffaele Mascella}
\address{Dipartimento di Scienze della Comunicazione, Campus Aurelio Saliceti, Università degli Studi di Teramo, Via R. Balzarini, 1, 64100, Teramo (TE).}
\email{dfazio2@unite.it, rmascella@unite.it}
\begin{document}
\maketitle
\begin{abstract}
This work explores Everett John Nelson's connexive logic, outlined in his PhD thesis and partially summarized in his 1930 paper \emph{Intensional Relations}, which is obtained by extending the system 
$\nl$ (reconstructed by E. Mares and F. Paoli) with a weak conjunction elimination rule explicitly assumed in the former but not in the latter. After a preliminary analysis of Nelson's philosophical ideas, we provide  an algebraic-relational semantics for his logic and we investigate possible extensions thereof which are able to cope with Nelson's ideas with much more accuracy than the original system. For example, we will inquire into extensions whose algebraic-relational models are endowed with irreflexive incompatibility relations, or determine a ``weakly'' transitive entailment. Such an investigation will allow us to establish relationships between some of the trademarks of Nelson's thought and concepts of prominent importance for connexive logic, as e.g. Kapsner's strong connexivity and superconnexivity, as well as between the algebraic-relational semantics of Nelsonian logics and ordered structures that have gained great attention over the past years, namely partially ordered involutive residuate groupoids and (non-orthomodular) orthoposets.  
\end{abstract}
\section{Introduction}\label{intro}
Connexive logic is nowadays an established field of research aimed at investigating well-motivated logical systems in the language of Classical Logic (or some expansion thereof) validating some contra-classical theses expressing a (necessary) \emph{connexion} or \emph{coherence} between antecedents (premises) and consequents (conclusions) of sound conditionals (valid inferences). Characteristic theses and conditions for a logic endowed with a binary connective $\to$ for implication and a unary connective $\neg$ for negation to be entitled as fully connexive are the following:
\begin{itemize}
\item[AT1] $\neg(A\to\neg A)$;
\item[AT2] $\neg(\neg A\to A)$;
\item[BT1] $(A\to B)\to\neg(A\to\neg B)$;
\item[BT2] $(A\to\neg B)\to\neg(A\to B)$.
\item[NSY] $(A\to B)\to(B\to A)$ is \emph{not} a theorem (for some $A,B$).
\end{itemize}
(AT1) and (AT2), called \emph{Aristotle's theses}, have a rather fair meaning: no proposition can entail its own negation. (BT1) and (BT2), called \emph{Boethius' Theses}, provide an axiomatic rendering of the fact that no formula $A$ can entail both $B$ and its negation, for any $B$. (NSY), that we may call the \emph{non-symmetry condition}, is aimed at ensuring that $\to$ can be understood as a genuine implication, rather than as an equivalence. This avoids the possibility of considering trivial limit cases (think e.g. to Classical Logic endowed with the usual material equivalence) as proper examples of a connexive logic.\\
Connexive theses have been motivated by different considerations depending on the specific meaning given to implication. \cite{cant2008} argues that connexive systems can formalize indicative natural language conditionals. 
Some authors (e.g. \cite{kapom2017}) suggest that a connexive implication is suitable for modelling counterfactual conditionals (see also \cite{wansinghunter}).  Moreover,  \cite{mccall75} proposes an interpretation of connexive conditionals in terms of physical or ``causal'' implications. 
Also, the results of empirical research on the interpretation of negated conditionals (see \cite{mccall2012,Pf2012,Pf2017,Pf2025,wansingstanford}) suggest that speakers having no previous knowledge of formal logic are inclined to consider connexive conditionals as the sound ones. Some attempts of explaining such evidences have been provided by A. Kapsner in \cite{Kaps2020} within a pragmatic framework. We refer the interested reader to \cite{wansingstanford} for a thorough overview of the topic.\\
A further motivation for endorsing a connexive implication, actually, the one we will be concerned with in this work, comes from the idea that semantic consequence is a content relationship. Over the past years, the need to provide an account of connexive implications  expressing (intensional) relations grounded on meaning rather than on truth values has been advocated by several researchers. On the one hand, a certain amount of literature has been devoted to highlight the possibility of defining well-behaving connexive implications within modal expansions of Classical Logic with an eye to highlight that connexive theses show up whenever strict conditionals (in the sense of C.I. Lewis, see e.g. \cite{Priest}) constrained by imposing further modal conditions on their arguments are taken into account. \cite{Gherardi,Pizzi,PizziWilliamson} are all cases in point. On the other hand, since its modern formulation due to Storrs McCall \cite{MacCall1,MacCall2}, the concept of a connexive implication has been often motivated in terms of \emph{logical incompatibility} between antecedents and consequents of sound conditionals. Indeed, as McCall explains in his PhD thesis \cite{MacCall1}, the notion of a connexive implication arises whenever a conditional is considered sound provided that the negation of its consequent is \emph{intensionally incompatible} with its antecedent. Such a definition dates back to Chrysippus, as passed down by Sextus Empiricus, and highlighted  e.g. by \cite{Kneale}. However, its modern revival has been proposed in the first half of the twentieth century by one, maybe the most unorthodox, of C. I. Lewis' students: Everett J. Nelson. Actually, as McCall points out, Nelson's account of conditionals is somewhat paradigmatic\footnote{See also \cite{Lenzen1} for a further interpretation of Nelson's role in the history of connexive logic.} when it comes to define what a connexive implication is: 
\begin{quote}
It has been given
the name of `connexive' implication in virtue of the fact that a true
conditional of this sort asserts, or comes to be in virtue of, a logical connexion or relation of incompatibility between the antecedent and the
negation of the consequent. [...] `connexive' implication turns out to be identical with what Nelson in 1930 called `intensional' implication, but for the benefit of those philosophers who
boggle at the adjective `intensional' [...] I use `connexive'. \cite[p. 9]{MacCall1} 
\end{quote}
\noindent Nelson outlines his perspective about implication in his PhD thesis \cite{Nelson3}, subsequently summarized in \cite{Nelson1}, and revised in \cite{Nelson2}. His approach to conditionals is intensional, as it aims at grounding the concept of an implication, that he calls \emph{entailment} (to distinguish it from the material conditional), on \emph{meaning} of its arguments. Indeed, {Nelson argues that} the truth of {many instances of} material and strict implication rests essentially on \emph{extrinsic} (extensional) or,  borrowing the Kantian lexicon {he employs in \cite[Ch. IV]{Nelson3}},   \emph{synthetic} relations between truth and, in the latter case, modal statuses of their antecedents and consequents, {regardless of  mutual relationships determined by what they actually state}. And this, in Nelson's view, is exactly the source of paradoxes of entailment that e.g. Lewis attempts rather unsuccessfully to avoid. Nevertheless, if one considers implication as it is used in ordinary reasoning, they recognize that what keeps arguments of a sound conditional together is in fact an \emph{intrinsic} or, again in Kantian terms, \emph{analytic} intensional relation between meanings. Specifically, sound conditionals are those whose consequent is \emph{contained} into the antecedent, in the same sense we mean whenever we say that the conclusion of a valid syllogism is contained into its premises.\\
For the sake of self-containment, we present a summary of Nelson's view of entailment and main motivations he adduces to support it below, see Section \ref{sec:NLanditssema}, and  \cite{MarPa,Nelson1} for details and precise references.

{Nelson defines sound conditionals as those whose antecedent is \emph{inconsistent}, or \emph{incompatible}, with the negation of its consequent. In fact, this definition is not a novelty introduced by Nelson's thought. It had already been discussed in depth by Lewis\footnote{See e.g. \cite[Ch. 5, p. 293]{Lewis1}.} and B. Russell\footnote{See e.g. \cite[Ch. 14, p. 148]{Russell1}.} in order to provide a definition of truth conditions of strict and material implication, respectively. Nevertheless, one of the main contributions to logic of Nelson's investigation is his deep revision of the notion of compatibility and its consequences. Indeed, as the author argues, the lack of a correct understanding of this concept turns out to be one of the main sources of paradoxes of entailment.}\\ Russell defines two propositions $p$ and $q$ being incompatible provided that they are not both true (``not-$p$ \emph{or} not-$q$''). And this yields a definition of ``$p$ \emph{implies} $q$'' as ``not-$p$ \emph{or} $q$'', from which one regains the usual definition of  material implication. Lewis strengthens Russell's definition upon setting $p$ and $q$ incompatible if they are not \emph{compossible} (namely, ``it is \emph{not} possible that \emph{both} $p$ and $q$''). And this leads him to identify ``$p$ \emph{implies} $q$'' with ``it is \emph{not} possible that \emph{both} $p$ and not-$q$'', i.e. with strict implication. As Nelson observes, Russell's and Lewis' notions of incompatibility fall short of taking into account meanings of propositions involved but only their truth values or modal statuses to the effect that there are propositions like e.g. contradictions that, being always false, and so  impossible, are incompatible with any proposition, even with themselves. {In particular, although Lewis recognizes the ``intensional'' nature of modalities, his definition of implication (and so of compatibility) is still strongly dependent on truth-functional conditions carried by connectives it involves, just as material implication is (see \cite[p. 63]{Woods1}).} This is, in Nelson's opinion, not acceptable. Compatibility is a relational concept depending on the meaning of \emph{both} propositions involved\footnote{In a very recent paper W. Lenzen has investigated Nelson's concept of consistency in the realm of modal logic. We refer to \cite{Lenzen} for details.}. Specifically, $p$ and $q$ are incompatible provided that they are mutually exclusive as the former ``contains'' in its meaning the \emph{proper} contrary (i.e. the negation) of the latter (and \emph{viceversa}), namely not-$q$ can be obtained from $p$ (not-$p$ can be obtained from $q$) by logical analysis. And this boils down to define a conditional of the form ``$p$ implies $q$'' sound provided that the meaning of $q$ is contained into the meaning of $p$.\\ 
As Nelson argues in several venues, once considered from {an intensional perspective}, no proposition, even a contradiction in a truth functional sense, does not clash with itself. In other words, propositions are always compatible with themselves. Even more, there are propositions as e.g. ``$2+2\ne 4$'' and ``Napoleon defeated Wellington'' which are compatible with each other although both impossible. \\
The above intuitions are exactly what leads Nelson to assume  connexive principles (although he does not have in mind a true concept of connexivity) as natural \emph{desiderata} when it comes to define the behavior of entailment. Indeed, {if one defines ``$p$ entails $q$'', provided that $p$ is incompatible with not-$q$ in an intensional sense, then since \emph{any} $p$ is always compatible with itself (or, alternatively, does not contain its own negation as part of its meaning) then Aristotle's Theses
\begin{center}
``It is not the case that $p$ implies not-$p$'' and ``It is not the case that not-$p$ implies $p$''
\end{center}
must hold.} Furthermore, since no proposition contains contradictory meanings/propositions, then ``$p$ implies $q$'' must always exclude, and so being always incompatible with ``$p$ implies not-$q$''. This boils down to the validity of Boethius' Theses: ``if $p$ implies $q$, then $p$ does not imply not-$q$'' and  ``if $p$ implies not-$q$, then $p$ does not imply $q$''.\\
{As D. J. Bronstein has observed in \cite[p. 167]{Bronstein}, the assumption that any proposition is compatible with itself might hint that, for Nelson, there can be no contradictory proposition at all. Indeed, how can one deny that ``$p$ \emph{and} not-$p$'' expresses (and therefore entails) both $p$ and not-$p$, and is thus self-contradictory (incompatible with itself)? Or, more in general, how can one disagree with the idea that  the conjunction of two incompatible propositions, think e.g. to ``$2<1$ \emph{and} $1=2$''\footnote{We thank an anonymous referee for having provided us with this striking example.} is self-contradictory, as it contains in its meaning mutually exclusive statements?} Nelson's solution involves a deep rethinking of the relationships among entailment, conjunction, disjunction, and negation, as well as of the inference rules long considered hallmarks of correct reasoning.

Classical logic attributes to conjunction an extensional meaning to the effect that an expression of the form ``$p$ and $q$'' is to be meant as ``$p$ is true and $q$ is true''. This, Nelson argues, is, in general, not correct. Indeed, the classical account considers conjunctions as mere ``aggregates'' of propositions. Of course, this is true, but only in case a conjunction is taken ``alone'', i.e. \emph{it is not in the antecedent of an entailment}. In the latter case, a conjunction expresses more than its parts, since it behaves like a whole, namely its arguments  ``function'' together in assessing the consequent. In other words, $r$ is entailed by ``$p$ and $q$'' if it results from the concrete combination of meanings of both $p$ and $q$ in such a way that none of them can be eliminated to the effect that $r$ is still obtainable from the antecedent. {Therefore, even if ``$p$ and not-$p$'' is ``materially'' contradictory, in the sense that it materially implies both $p$ and not-$p$, i.e. two incompatible  propositions and, \emph{a fortiori} its own negation, it is nevertheless not contradictory from an intensional perspective, since, \emph{as a unit} (i.e. in the antecedent of an entailment, cf. \cite{Nelson4}) it does not entail them. A similar argument applies e.g. to ``$2<1$ \emph{and} $1=2$''.}\\ {Actually, in his late reply to Bronstein's objections \cite{Nelson4}, Nelson seems to admit that there might be (``if any'', his words) propositions having incompatible consequences (in his sense), and so contradictory. However,  lacking examples or additional details, his response appears to be a somewhat instrumental effort to ``end the discussion'' rather than an elucidation. }

As a consequence of his perspective {about conjunction}, Nelson  categorically rejects ({although his position becomes more moderate in his later investigations, see \cite{Nelson4} and p.\pageref{nelmodpos} below}) the axiom of conjunction simplification 
\begin{center}
``if $p$ and $q$, then $p$'' 
\end{center}
(for \emph{any} $p$ and $q$) and so of ``if $p$ and not-$p$, then $p$''. Consequently, he rejects also the usual axiom 
\begin{center}
``if $p$ entails $q$ \emph{and} $q$ entails $r$, then $p$ entails $r$'' 
\end{center}
stating the ``transitivity'' of implication. Indeed, assuming the latter would force us to accept any of its instances and so also  
\begin{center}
``if $p$ entails $p$, \emph{and} $p$ entails $r$, then $p$ entails $r$''. 
\end{center}
This conflicts with Nelson's ideas about conjunction as the first conjunct in the antecedent adds nothing to the second in order to assess the consequent (cf. \cite{Nelson4}).\\
Similar arguments applied to show the non validity of conjunction' simplification on intensional grounds lead Nelson to reject the axiom of disjunction introduction
\begin{center}
``if $p$, then $p$ \emph{or} $q$''.
\end{center}
{Indeed, as Nelson argues, from the \emph{truth} of $p$ one can always deduce the \emph{truth} of ``$p$ or $q$'', since, once regarded as statements about \emph{classes}, if e.g. $p$ is  ``$x\in A$'' and $q$ is ``$x\in B$'', then $p$ carries more information than ``$x\in A$ \emph{or} $x\in B$'' ($p$ or $q$). Therefore, whenever the former holds true, the latter does (see \cite[p. 273]{Nelson2}, cf. \cite[p. 448]{Nelson1}). However, Nelson claims, once considered from an intensional perspective, ``$p$ or $q$'' carries information that cannot be inferred from $p$ alone. Indeed, in \cite[p. 448]{Nelson1} Nelson writes:
\begin{quote}
``$p$ entails $p$ or $q$'' cannot be asserted on logical grounds, because from an analysis of $p$ we cannot derive the propositional function ``$p$ or $q$'' where $q$ is a variable standing for just any other propositional function whatsoever.  
\end{quote}
}
As observed by E. Mares and F. Paoli in \cite{MarPa}, Nelson's  ``radical stance'' about the relationship between conjunction and entailment might suggest (although there is not enough evidence to attribute such a view to him) that Nelson has in mind a rather strong  notion of real use in proof. Indeed, it can be argued that, in Nelson's view, real use in proof could be somehow considered as the  ``pragmatic'' rendering of meaning entailment, in the sense that $r$ is entailed by $\{p_{1},\dots,p_{n}\}$ ``in force of their meaning'' if and only if the meaning of $r$ is obtained/produced upon combining (``practically'') the meanings of \emph{all} the $p_{i}$'s.\\ A further feature of Nelson's thought supporting this line of interpretation is his acceptance of the transitivity axiom 
\begin{center}
``if $p$ entails $q$, then, if $q$ entails $r$, $p$ implies $r$''.
\end{center}
\emph{whenever $p,q,r$ are mutually intensionally non-equivalent}. In this case, as Nelson argues, the application of such an axiom would  yield an effective (and ``irredundant'') gain from a logical perspective.
 
To provide a rigorous treatment of intensional logical relations, Nelson outlines a set of seven postulates (axioms) encoding his view of logical notions, and inference rules which he borrows from Whitehead and Russell's \emph{Principia Mathematica}. However, as he premises in several venues, his aim is not to furnish a complete system of logic but just to treat intensional notions formally. In fact, he has even no evidence as to whether his set of postulates and inference rules is consistent or not.\\
Actually, if we confine ourselves to take into account his published works \cite{Nelson1,Nelson2}, piecing together Nelson's  ``original'' system seems a rather hard task. Indeed, in presenting his ``system'' in \cite{Nelson1}, Nelson outlines his postulates, but does not state inference rules he applies for obtaining his results clearly, confining himself to talk rather vaguely about ``a principle of inference and the method of proof employed in the \emph{Principia Mathematica}'' (see \cite[p. 449]{Nelson1}). Although the ``principle of inference'' he refers to is clear, since it coincides with Modus Ponens\footnote{Indeed,  ``the process of inference'' is how Whitehead and Russell call Modus Ponens in \cite[p. 8]{PrincipiaMath1}.}, less evident is which of the various methods of proof contained in Whitehead and Russell's work (see \cite{Oreily} for an account) should be included in his system. Moreover, further problems concerning the notation arise. Indeed, following the notational convention from \emph{Principia Mathematica}, Nelson represents long chains of conjunctions omitting brackets. As already remarked by \cite{MarPa}, this might hint that Nelson means conjunction as associative but there is not enough evidence about it. Moreover, he uses the same convention but makes no (at least explicit) reference to associativity not even in his PhD thesis, where a discussion of properties of connectives he assumes is carried out at length.   

The first and, to the best of our knowledge, only attempt to regard Nelson postulates and inference rules as a true formal system has been carried out by E. Mares and F. Paoli in \cite{MarPa} where, beside providing a historical analysis of Nelson's logic in relationship with Lewis' school, they outline a Hilbert calculus $\nl$ which extends Nelson's axioms with inference rules for Modus Ponens, Adjunction, uniform instantiation of propositional variables and replacement of equivalents. Providing a semantics for their system has been an open problem until very recent times. See \cite{wansingstanford} and \cite{Wans3}.
 
In this paper, we investigate Nelson' system of logic as outlined in his PhD thesis \cite{Nelson3}, partially summarized in \cite{Nelson1}, and revised in \cite{Nelson2}. It will turn out that the system $\nl$ originally regained by Mares and Paoli, once presented as Nelson does in his first work, namely with the addition of a weak inference rule for conjunction elimination (see Section \ref{sec:NLanditssema}),  is sound and complete with respect to a class of structures consisting of an algebra endowed with primitive operations which are the semantical counterpart of conjunction, negation and compatibility connectives, an external binary ``incompatibility'' relation over it, and two external constants/truth values. Starting from the basic system here called  (we admit, a little inappropriately) again $\nl$ and its extension $\nel$ with stronger inference rules obtained upon eliminating the constraint that premises be theorems, we delve into an investigation of some further  extensions of $\nel$ and $\nl$ which, in our opinion, are worth of attention since they are arguably able to cope with Nelson's perspective on logic with better accuracy than $\nl$. For example, in the light of the acquired semantics, we will deal with the problem of obtaining a logic whose non-trivial models are endowed with an irreflexive incompatibility relation. It will turn out that such an investigation allows us to establish a link between Nelson's  view and some concept which have been introduced in the realm of connexive logic investigations only in recent times, namely Kapsner' \emph{strong connexivity} and \emph{superconnexivity} \cite{Kap1}. Indeed, as we will see, an extension of Nelson's logic is sound and complete w.r.t. a class of models endowed with an irreflexive incompatibility relation if and only if, on the syntactical side, it allows to derive a weak superconnexive thesis encoding the idea of treating the failure of Aristotle's Theses as explosive. Moreover, such an extension will be also strongly connexive if and only if it allows to derive a further weakly superconnexive thesis establishing the ``explosiveness'' of the failure of Boethius' Theses.\\  Furthermore, we will discuss extensions of Nelson's logic(s) with inference rules whose acceptance by Nelson can be deduced from his writings. Specifically, upon considering the extension of $\nl$ and $\nel$ with a rule of weak disjunction introduction, it will turn out that such a logic satisfies the principle of explosion (\emph{ex contradictione quodlibet}) and so, as it will be shown, is strongly connexive and  weakly superconnexive.\\
Along the paper we will also provide some reflection on the non-symmetry of entailment as well as on \emph{hyperconnexivity} (see \cite{mccall2012}). It will turn out that ``Nelsonian'' logics determine an entailment relation which is non-symmetric and non-hyperconnexive in a strong sense. Indeed, \emph{any} model of these logics provides a counterexample to symmetry as well as to hyperconnexivity on pain of triviality. Such a property is, to the best of our knowledge, not very common among connexive logics outlined in the literature.\\
In the last part of the paper, we will introduce a further extension of Nelsonian logics by means of a weak inference rule for the transitivity of implication and we will argue for its compatibility with Nelson's philosophy. Among interesting features this logic  enjoys, it will turn out that its models yield incompatibility relations which allow to define a partial order turning their algebraic reduct in a commutative residuated groupoid with antitone involution having an underlying poset with antitone involution which is never a lattice, can  always be completed (via its Dedekind-MacNeille completion) to an ortholattice that, in turn, is never neither Boolean nor an orthomodular lattice.

The research carried out in this work does not claim to be  exhaustive. This is witnessed, among other things, by open problems which are from time to time stated along the next pages. Nor it has been conceived to be an attempt of defending Nelson's philosophical conceptions and, \emph{a fortiori}, {his justification of connexive principles}. Our primary goal is to establish a foundation for a systematic investigation of Nelson's system(s), which, at least in our opinion, have received far too little attention in recent years.    

This work is organized as follows. In Section \ref{prel} we dispatch some basic concepts from Abstract Algebraic Logic that will be expedient for the development of our arguments. Section \ref{sec:NLanditssema} will be devoted to outline Nelson' system and some slight variant thereof. In Section \ref{sec:algsemnl} we will investigate the algebraic relational semantics of Nelson systems presented in the previous section. Some remarks on the relationship between our models and the theory of logical matrices will be provided. In Section \ref{sec:variantsNL} and Section \ref{sec:n1} we will discuss some further extension of Nelsonian logics with inference rule for the irreflexivity of incompatibility and the weak transitivity of entailment. We conclude in Section \ref{sec:concl}.

\section{Preliminaries}\label{prel}
In this section we provide an (of course, not exhaustive) overview of basic concepts from Abstract Algebraic Logic that will turn out to be useful for the development of our arguments. We refer the interested reader to \cite{Font16}  for details.
\subsection{Some algebraic logic}
Let $\mathcal{L}=\{f_{i}\}_{i\in I}$ be an algebraic language of a fixed type, and let $\mathbf{Fm}_{\mathcal{L}}$ be the absolutely free algebra in the language $\mathcal{L}$ generated by an infinite countable set $Var$ of variables. A \emph{consequence relation} on  $\mathrm{Fm}_\mathcal{L}$ is a relation $\vdash\ \subseteq \mathcal{P}(\mathrm{Fm}_{\mathcal{L}}) \times \mathrm{Fm}_{\mathcal{L}}$ such that, {for all $\Gamma,\Delta\in\mathcal{P}(\mathrm{Fm}_{\mathcal{L}})$ and $\varphi\in \mathrm{Fm}_{\mathcal{L}}$,}
	
	\begin{enumerate}
		\item[R.] If $\varphi \in \Gamma$, then $\Gamma \vdash \varphi$
		\item[M.] If $\Gamma \vdash \varphi$ and $ \Gamma\subseteq \Delta$, then $\Delta \vdash \varphi$
		\item[C.] If $\Gamma \vdash \varphi$ and $\Delta \vdash \psi$ for all $\psi \in \Gamma$, then $\Delta \vdash \varphi$.
	\end{enumerate}
	
\noindent	A \emph{sentential logic} (a logic, in brief) $\mathsf{L}$ (see \cite[Definition 1.5]{Font16}) is a consequence relation $\vdash_{\mathsf{L}} {}{} \subseteq \mathcal{P}(\mathrm{Fm}_{\mathcal{L}}) \times \mathrm{Fm}_{\mathcal{L}}$, which is \emph{substitution-invariant}  in the sense that, for every substitution (i.e. endomorphism) $\sigma \colon \mathbf{Fm}_{\mathcal{L}} \to \mathbf{Fm}_{\mathcal{L}}$,
	\[
	\text{if }\Gamma \vdash_{\mathsf{L}} \varphi \text{, then }\sigma \Gamma \vdash_{\mathsf{L}} \sigma \varphi.
	\]
Whenever the reference to $\mathsf{L}$ will be clear from the context, we will write $\vdash$ in place of $\vdash_{\mathsf{L}}$. Using a customary notation, a logic  $\vdash_{\mathsf{L}}$ over a language $\mathcal{L}$ will be henceforth denoted by the pair $\mathsf{L}=\langle\mathbf{Fm}_{\mathcal{L}},\vdash_{\mathsf{L}}\rangle$. Moreover, if $\Theta\subseteq\mathcal{P}(\mathrm{Fm}_{\mathcal{L}})\times\mathrm{Fm}_{\mathcal{L}}$, we denote by $\mathsf{L}+\Theta$ the \emph{extension} of $\mathsf{L}$ axiomatized by $\Theta$, i.e. the \emph{least} logic $\mathsf{L}'=\langle\mathbf{Fm}_{\mathcal{L}},\vdash_{\mathsf{L}'}\rangle$ such that $\vdash_{\mathsf{L}}\ \subseteq\ \vdash_{\mathsf{L'}}$ and all elements of $\Theta$ are inference rules of $\mathsf{L'}$. 

Let $\mathsf{L}=\langle\mathbf{Fm}_{\mathcal{L}},\vdash_{\mathsf{L}}\rangle$ be a sentential logic over a language $\mathcal{L}$, $\Gamma\subseteq\mathrm{Fm}_{\mathcal{L}}$, $\alga$ an algebra in the language $\mathcal{L}$ (an $\mathcal{L}$-\emph{algebra}, for short), $h\in\mathbf{Hom}(\mathbf{Fm}_{\mathcal{L}},\alga)$. We set
$$h(\Gamma):=\{h(\gamma):\gamma\in\Gamma\}.$$ 
A \emph{matrix} over $\mathcal{L}$ (or $\mathcal{L}$-matrix) is a pair $\langle\alga,F\rangle$ such that $\alga$ is an algebra in the language $\mathcal{L}$ and $F\subseteq A$ is a set, called a \emph{filter}, of \emph{designated} elements of $\alga$. If $\mathsf{L}=\langle\mathbf{Fm}_{\mathcal{L}},\vdash_{\mathsf{L}}\rangle$ is a logic and $\mathcal{M}=\langle\alga,F\rangle$ is an $\mathcal{L}$-matrix, we say that $\mathcal{M}$ is a \emph{matrix model} of $\mathsf{L}$ provided that, for any $\Gamma\cup\{\varphi\}\subseteq\mathrm{Fm}_{\mathcal{L}}$, and for any homomorphism  $h\in\mathbf{Hom}(\mathbf{Fm}_{\mathcal{L}},\alga)$, it holds that
\[h(\Gamma)\subseteq F\quad\text{ implies }\quad h(\varphi)\in F.\]
In this case, $F$ is called an $\mathsf{L}$-filter.

Let $\langle\alga,F\rangle$ be an $\mathcal{L}$-matrix and let $\theta\subseteq A^{2}$ be a congruence over $\alga$, we say that $\theta$ is \emph{compatible} with $F$ provided that, for any $a,b\in A$, $a\in F$ and $a\theta b$ implies $b\in F$. It is well known that, for any $\mathcal{L}$-matrix $\mathcal{M}=\langle\alga,F\rangle$, the set $\mathrm{Comp}^{F}\alga$ of  congruences over $\alga$ compatible with $F$ forms a complete sublattice of the congruence lattice of $\alga$ w.r.t. set inclusion. Consequently $\mathrm{Comp}^{F}\alga$ admits a maximum element usually denoted by $\Omega^{\alga}(F)$ and called the \emph{Leibniz congruence} of $\mathcal{M}$. If $\mathcal{M}=\langle\alga,F\rangle$ is a matrix model of the logic $\mathsf{L}=\langle\mathbf{Fm}_{\mathcal{L}},\vdash_{\mathsf{L}}\rangle$, we say that $\mathcal{M}$ is \emph{reduced} provided that $\Omega^{\alga}(F)=\Delta$, where $\Delta=\{(a,a):a\in A\}$. The class of all reduced models of a logic $\mathsf{L}$ is customarily denoted by $\mathsf{Mod}^{*}\mathsf{L}$. The class of algebraic reducts of members of $\mathsf{Mod}^{*}\mathsf{L}$ is denoted by $\mathsf{Alg}^{*}\mathsf{L}$. It is well known that any logic is complete w.r.t. the class of all its reduced matrix models, see \cite[Theorem 4.43]{Font16}.\label{leibniz}   

\section{E. J. Nelson's logic and some variants thereof}\label{sec:NLanditssema}
This section is devoted to exploring Nelson's connexive logic as it is presented in \cite{Nelson3} and partially summarized in \cite{Nelson1}, and some variants thereof obtained by strengthening some inference rules Nelson assumes. In Section \ref{sec:variantsNL}, we will consider extensions of logics outlined here with inference rules that are not included explicitly in Nelson's formal system, but whose validity is explicitly or implicitly endorsed by the author in his works. 

{We will present Nelson's system as a calculus of theorems. Therefore, as it will be seen in a while, inference rules of Nelson's calculus apply only to provable sentences.\\
This choice is coherent with \cite{MarPa} and further motivated by the fact that Nelson's aim is to provide a calculus of \emph{assertions} as propositions \emph{certified on logical grounds}, that, according to Lewis' and Russell's notation, namely to his main bibliographical references, see \cite{Lewis1} and \cite{Russell1,PrincipiaMath1}, respectively, he denotes by $\vdash\varphi$, where $\varphi$ is a proposition. Now, due to such a notation, the contemporary reader would be led to identify assertions with theorems, namely with propositions/formulas which are \emph{provable} within a logical system by making use of axioms and inference rules thereof without relying on additional assumptions.   However, at the time in which Nelson writes his first works \cite{Nelson3,Nelson1}, the concept of an assertion is far from being fully transparent.\\ In \cite[p. 282]{Lewis1}, Lewis states: ``the idea of assertion is just what would be supposed -- a proposition may
be asserted or merely considered''. Here, Lewis uses the term ``assertion'' to distinguish occurrences of propositions/formulas which are merely considered as e.g. in ``$\varphi$ is a proposition'', or that appear as subformulas of other formulas, from propositions which have logical significance. In other words, an assertion is any proposition which is somehow ``stated'' (as a truth) and ``judgeable''. Although some of his examples may look quite  misleading (``An asserted propositional function is such as  `A is A' where A is undetermined'', \cite[p. 282]{Lewis1}), Lewis notion of an assertion seems rather independent on the notion of a theorem or of a tautology. And actually, along the pages of \cite{Lewis1}, we find references to assertions which can be either true \emph{or false} (cf. e.g. \cite[p. 214]{Lewis1}). \\  Whitehead and Russell's notion of an assertion, at least as presented e.g. in \cite{PrincipiaMath1} (see also \cite{Kelly1,Russell1}) can be rendered by means of what follows:
\begin{quote}
The sign ``$\vdash$'', called the ``assertion-sign'', [...] is required for distinguishing a complete proposition, which we assert, from any subordinate propositions contained in it but not asserted. In ordinary written language a sentence contained between full stops denotes an asserted proposition, and if it is false the book is in error. [...] For example, if ``$\vdash p\supset p$'' occurs, it is to be taken as a complete assertion convicting the authors of error unless the proposition ``$p\supset p$'' is true (as it is). \cite[p. 8]{PrincipiaMath1}
\end{quote}
As the reader may recognize from the above passage, the notion of ``asserted'' as proposed by Whitehead and Russell may be arguably rendered, although the example they propose is rather unfortunate, as ``claimed to be true'' and so as either an assumption or as the output of an inference.\\ By defining an assertion as a proposition which is ``certified on logical grounds'', Nelson seems to assume a strengthened version of Whitehead and Russell's notion.  However, at least standing to \cite[Ch. VI.2]{Nelson3}, where Nelson discusses the topic at length, his treatment of the concept seems rather convoluted (``E'' stands for entailment, see below):  
\begin{quote}
``$\vdash p E p$'' does mean that the denial of ``$p E p$'' is irrational. For convenience sake, ``$\vdash p E p$'' may be read `` `$p E p$' is true'', or `` `$p E p$' is necessary'', or simply, `` `$p E p$' is asserted''. \cite[Ch. VI.2, p. 72]{Nelson3}  
\end{quote}
If one confines oneself to considering the passage above, one is tempted to attribute to ``assertion'' the status of a  necessary logical truth, or of something which is provable. However, a few lines further down, Nelson is careful to specify that what is certified on logical grounds need not be a logical law (\cite[p. 72]{Nelson3}) but can nevertheless be obtained as an instance of a logical law.\\ It can be argued that Nelson's original intention is to borrow  Whitehead and Russell's definition (which he mentions at the very beginning of \cite[Ch. VI.2]{Nelson3}), although his rendering results in a strengthened version thereof. Nevertheless, whether his notion is the result of a misinterpretation of Whitehead and Russell's thought, or of a reckless explanation, or intentional, we cannot say. Therefore, in absence of ultimate evidences about the way he means it, interpreting Nelson's concept of an assertion (at least as he employs it in his early investigations) needs to rely on a good deal of speculation.}\\
In any case, what can be said with a certain degree of certainty is that Nelson's main goal is not to provide a calculus of tautologies, namely formulas which are true regardless of truth values assigned to variables they contain, but rather of ``propositions asserting intensional relations between propositions'', or, as he means them borrowing Kant's perspective,  ``analytic propositions'' (see \cite[Ch. IV]{Nelson3}). An analytic or intensional relation between propositions is a relation holding between propositions in force of their \emph{meaning}, and so necessary and independent on their truth values.
 
We present Nelson's calculus and logic as follows. Let $Var$ be an infinite countable set of variables and let $\mathbf{Fm}_{\mathcal{L}}$ be the absolutely free algebra generated by $Var$ in the language $\mathcal{L}=\{\otimes,\circ,{}^{*}\}$, where  $ar(\otimes)=2=ar(\circ)$, and $ar({}^{*})=1$. $\otimes$, $\circ$ and ${}^{*}$ will stand for (intensional) conjunction\footnote{Note that,  while \cite{MarPa} denotes conjunction with the more traditional $\land$, here we use (the reminiscent of linear logic) $\otimes$ in order to avoid the reader to be tempted of attributing to it any property of its classical counterpart.}, ``compatibility'' and negation, respectively.\\
Let us set $\varphi\ra\psi:=(\varphi\circ\psi^{*})^{*}$, $\varphi\Leftrightarrow \psi:=(\varphi\ra\psi)\otimes(\psi\ra\varphi)$,   $\varphi\not\Leftrightarrow\psi:=(\varphi\Leftrightarrow\psi)^{*}$, and \[\varphi\nlra\psi\nlra\chi:=((\varphi\nlra\psi)\otimes(\varphi\nlra\chi))\otimes(\psi\nlra\chi).\]
$\mathsf{NL}$ is the pair $\langle\mathbf{Fm}_{\mathcal{L}},\vdash_{\mathsf{NL}}\rangle$, where $\vdash_{\mathsf{NL}}$ is the derivability relation (to be meant as customary) induced by the axiom and inference rule schemes in Figure \ref{fig:nelson}. Axioms (A1)-(A7) are explicitly assumed by Nelson in \cite{Nelson3,Nelson1}\footnote{It is worth noticing that, in outlining his logical system,  Nelson employs a notation which is, apart from the use of $\circ$ for compatibility/consistency, very different from our. Indeed, in his works $\Leftrightarrow$ ($\not\Leftrightarrow$) is expressed by $=$ ($\ne$), ${}^{*}$ is $-$, $\otimes$ is $.$ and (the term-defined) entailment $\ra$ is $E$ (to distinguish intensional entailment from material implication $\supset$). While Nelson introduces a specific symbol $/$ for inconsistency, we will not make use of it in the sequel. See e.g. \cite[p. 444--446]{Nelson1}, cf. \cite{MarPa}.}. Furthermore, Nelson states (MP) and (Adj) at once with the principle
\begin{quote}
If a proposition, or all the propositions of a logical product, are asserted, separately, or otherwise, but as such, then whatever is entailed by or is identical with that proposition, or that logical product, is categorically asserted. \cite[Item 1.91, pp. 141]{Nelson3} 
\end{quote}
We observe that the quotation expresses, among other things, the validity of the inference rule assumed by Whitehead and Russel in \cite{PrincipiaMath1}\footnote{See \cite{Oreily} for an useful overview and historical analysis of the propositional logic of Principia Mathematica.}
\begin{prooftree}
\AxiomC{$\vdash\varphi$}
\AxiomC{$\vdash\varphi\Leftrightarrow\psi$}\RightLabel{$\mathrm{(MP)}_{1}$}
\BinaryInfC{$\vdash\psi$} 
\end{prooftree}
However, it can be easily derived by means of (MP) and (CE), or by $\mathrm{(Eq)}_{1}$.\\
The validity of (CE) is expressed by the following principle:
\begin{quote}
Any proposition unconditionally asserted within a complex, is asserted separately; that is, outside the complex.\cite[Item 1.9, p. 141]{Nelson3}
\end{quote}
Indeed, as Nelson explains, this principle expresses the logical validity of ``the passage [...] from a material conjunction to one of the conjuncts; e.g., if $p$ and $q$ are both true, then we can assert $p$'' \cite[p. 142]{Nelson3}.\\
$\mathrm{(Eq)}_{1}$ is nothing but the formal rendering of the fact that ``any proposition may be substituted for any proposition that is intensionally equivalent to it'' \cite[p. v]{Nelson3}.\\
\begin{figure}
\fbox{
  \begin{minipage}[l]{\textwidth}
\begin{enumerate}[{A}1]
\item $\varphi\ra\varphi$
\item $(\varphi\circ\psi)\ra(\psi\circ\varphi)$
\item $\varphi\ra \varphi^{**}$
\item $(\varphi\ra\psi)\ra(\varphi\circ\psi)$
\item $(\varphi\otimes\psi)\Leftrightarrow(\psi\otimes\varphi)$
\item $((\varphi\otimes\psi)\ra\chi)\ra((\varphi\otimes\chi^{*})\ra\psi^{*})$
\item $(\varphi\nlra\psi\nlra\chi)\ra((\varphi\ra\psi)\ra((\psi\ra\chi)\ra(\varphi\ra\chi))$
\end{enumerate}
\begin{center}
\AxiomC{$\vdash\varphi\ra\psi$}
\AxiomC{$\vdash\varphi$}\RightLabel{(MP)}
\BinaryInfC{$\vdash\psi$}
\DisplayProof\quad
\AxiomC{$\vdash\varphi$}
\AxiomC{$\vdash\psi$}\RightLabel{(Adj)}
\BinaryInfC{$\vdash\varphi\otimes\psi$}
\DisplayProof\quad
\AxiomC{$\vdash\varphi\Leftrightarrow\psi$}
\AxiomC{$\vdash\chi$}\RightLabel{$\mathrm{(Eq)}_{1}$}
\BinaryInfC{$\vdash\chi'$} 
\DisplayProof,\\
\quad\\
\quad\\
\AxiomC{$\vdash\varphi\otimes\psi$}\RightLabel{(CE)}
\UnaryInfC{$\vdash\varphi$}
\DisplayProof,
\end{center}
where $\chi'$ has been obtained upon replacing one or more occurrence of $\varphi$ by $\psi$ in $\chi$.
  \end{minipage}

}
\caption{The formal system $\nl$}
\label{fig:nelson}
\end{figure}
Note that our presentation of $\mathsf{NL}$ lacks the rule of \emph{uniform substitution of propositional variables}, that Nelson and \cite{MarPa} assume explicitly while in this venue it is taken as implicit, since we work with axiom/inference rules schemas rather than with simple axioms and inference rules. Moreover, we observe that, in presence of $\mathrm{(CE)}$ and (Adj), ($\mathrm{A5}$) can be provably replaced by the axiom schema 
\[\varphi\otimes\psi\ra\psi\otimes\varphi,\tag{$\mathrm{A5}^{*}$}\label{a5bis}\]
as they are interderivable.\\
It is important to observe that Mares and Paoli's reconstruction of $\nl$ (see e.g. \cite{MarPa}) is slightly different from ours, since they do not assume (CE) and $\mathrm{(Eq)}_{1}$, and they take in place of the latter the weaker inference rule 
\begin{prooftree}
\AxiomC{$\vdash\varphi\Leftrightarrow\psi$}
\RightLabel{$\mathrm{(Eq)}$}
\UnaryInfC{$\vdash\chi\Leftrightarrow\chi'$} 
\end{prooftree}
where $\chi'$ is obtained upon replacing one or more occurrences of $\varphi$ by $\psi$ in $\chi$. {This is because their reconstruction of Nelson's system of logic relies primarily on his published article, namely \cite{Nelson1}, where Nelson does not clearly specify the inference rules his system employs and makes no mention of the (CE) rule. In his doctoral thesis, this rule is included among inference schemas of the system, whereas in the cited article, it is merely stated as an acceptable principle.}\\ Of course, in presence of (CE), the other inference rules of $\nl$ make $\mathrm{(Eq)}_{1}$ and (Eq) interderivable. Nevertheless, we are not aware if this still holds whenever (CE) is dropped. Therefore, we  believe that the following problem is worth of consideration
\begin{question}Prove or disprove that (Eq) and $\mathrm{(Eq)}_{1}$ are interderivable in presence of axioms (A1)-(A7) and inference rules (MP),(Adj) of $\mathsf{NL}$.
\end{question}
 
As recalled in the Introduction, in  \cite{Nelson3,Nelson1} Nelson refuses explicitly to consider the axiom of \emph{conjunction simplification} \[\varphi\otimes\psi\ra\varphi.\tag{S}\label{simpl}\] as a suitable assumption when it comes to formalize a logic of intensional entailment.  Indeed, in his view, \emph{when considered in relation with entailment}, the conjunction $\varphi\otimes\psi$ of $\varphi$ and $\psi$ has not to be considered as a mere ``aggregate'' of $\varphi$ and $\psi$, namely it should not be meant extensionally as ``$\varphi$ is true \emph{and} $\psi$ is true''. Instead, $\varphi\otimes\psi$ yields what can be obtained by the ``joint force'' of \emph{both} $\varphi$ and $\psi$, in the sense that $\chi$ is a consequence of $\varphi\otimes\psi$ if and only if $\chi$ is derivable from $\varphi$ and $\psi$ by using both of them \emph{non trivially}. In Nelson's words, ``they function together'' \cite[p. 275]{Nelson2}.\\ {It is important to emphasize that, despite some interpretations put forward in the past suggesting the contrary\footnote{``Nelson sees the difference between the intensional
`and' and the truth-functional `.' as a difference between two propositional connectives,
as a difference marking a difference between two kinds or breeds of conjunction.'' \cite[p. 68]{Woods1}}, Nelson does not consider intensional conjunction and extensional (namely, material) conjunction as two distinct concepts. For him, conjunction is one and the same, although it has a dual nature. It should be understood extensionally when considered in itself, and (always) intensionally when considered in the antecedent of an entailment. In the latter case, it functions, by way of metaphor,  as a comma separating premises of a syllogism, which, to be meaningful, must have a non-vacuous relationship with the conclusion.}

Of course, it naturally raises the question whether adding \eqref{simpl} to $\nl$ yields a consistent logic  but, unfortunately, the answer is negative. Indeed, $\nl +\{\mathrm{\eqref{simpl}}\}$ is \emph{trivial}. For reader's convenience, we outline the proof of this fact  in the following proposition (cf. \cite{Pizzi}).
  
\begin{proposition}\label{prop:trivial}The extension $\nl+\{\mathrm{S}\}$ is inconsistent.
\end{proposition}
\begin{proof}
The conclusion rests on the fact that, for arbitrary $\varphi,\psi\in\mathrm{Fm}_{\mathcal{L}}$, $\vdash_{\nl+\{\mathrm{S}\}}\varphi\otimes\varphi^{*}\ra\psi$ is granted by \eqref{simpl}, (A6), (MP), and $\mathrm{(Eq)}_{1}$, as $\vdash_{\nl+\{\mathrm{S}\}}\psi\ra\psi^{\ast\ast}$ and $\vdash_{\nl+\{\mathrm{S}\}}\psi^{**}\ra\psi$, and so $\vdash_{\nl+\{\mathrm{S}\}}\psi^{**}\Leftrightarrow\psi$, hold. Moreover, one has $\vdash_{\nl+\{\mathrm{S}\}}\varphi\otimes\varphi^{*}\ra\varphi^{*}$ and $\vdash_{\nl+\{\mathrm{S}\}}(\varphi\otimes\varphi^{*}\ra\varphi^{*})^{*}$ by (A4). Therefore, we conclude also $\vdash_{\nl+\{\mathrm{S}\}}(\varphi\otimes\varphi^{*}\ra\varphi^{*})\otimes(\varphi\otimes\varphi^{*}\ra\varphi^{*})^{*}$ from which the desired conclusion easily follows. 
\end{proof}
{We observe that the proof of Proposition \ref{prop:trivial}  rests essentially on that \eqref{simpl} allows us to derive, together with (A6), the formula (a) $\varphi\otimes\varphi^{*}\ra\psi$ (for any $\varphi,\psi\in\mathrm{Fm}_{\mathcal{L}}$) which, as Nelson explains in his late note \cite{Nelson4}, is precisely the reason why he assumes a so strong position against conjunction simplification. Indeed, in response of D. J. Bronstein's criticisms of his notion of entailment \cite{Bronstein}, Nelson writes (notation has been changed according to the present work): 
\begin{quote}
Mr. Bronstein is right in saying that I did not assert $p\otimes q\ra p$ because I wanted to avoid such paradoxes as $p\otimes p^{*}\ra q$. During the six or seven years after I wrote that paper I have changed my mind on that point. I still believe, however, that $p\otimes p^{*}\ra q$ is unacceptable, but that it can be avoided by means other than the rejection of $p\otimes q\ra p$. 
\end{quote}
\label{nelmodpos}Among other things, this passage is expedient for highlighting that, several years after \cite{Nelson1},  Nelson's radical rejection of \eqref{simpl} becomes more moderate (cf. \cite[p. 412]{MarPa}). In fact, Nelson's attitude towards \eqref{simpl} starts to waver already in \cite{Nelson2}, where he proposes to replace (A6) with (the association of conjunctions is our, see below) \[(((\varphi\circ\psi^{*})\otimes(\chi\circ\psi^{*}))\otimes(\varphi\circ\chi^{*}))\otimes(\chi\circ\varphi^{*})\ra((\varphi\otimes\chi\ra\psi)\ra(\varphi\otimes\psi^{*}\ra\chi^{*})),\label{a6ast}\tag{$\mathrm{A6}^{*}$}\]
or the weaker \[(\varphi\circ\psi^{*})\otimes(\chi\circ\psi^{*})\ra((\varphi\otimes\chi\ra\psi)\ra(\varphi\otimes\psi^{*}\ra\chi^{*})),\]
in order to allow the presence of \eqref{simpl} and, at the same time, still avoiding the (at least, direct) derivability of (a), \cite[p. 281]{Nelson2}.\\
As the reader may readily observe, exploring modifications of $\nl$ according to the above ideas would require us to consider logical systems and logics significantly different from those we intend to analyze here. Therefore, given that they warrant separate, in-depth investigations, we defer their examination to future work.  
}

\label{probassocia}In the presentation of his axioms, Nelson does not make use of brackets to make explicit how conjunctions must be associated. This might hint that Nelson assumes (at least implicitly) intensional conjunction to be associative. However, as remarked by \cite[p. 418]{MarPa}, at least if one sticks to \cite{Nelson3,Nelson1,Nelson2}, there is not enough evidence to conclude  whether the omission of an associativity axiom is an oversight or is intentional.\\
Actually, by (CE) and (Adj), the following pair of inference rules is derivable in $\nl$:
\begin{prooftree}
\AxiomC{$\vdash(\varphi\otimes\psi)\otimes\chi$}
\doubleLine
\UnaryInfC{$\vdash\varphi\otimes(\psi\otimes\chi)$} 
\end{prooftree}
where the double line means that the inference rule can be applied top-bottom and the other way around. However, in no place of Nelson's work on intensional logic there is an explicit reference to how conjunctions have to be understood in antecedents (or consequents) of entailments.  Therefore, for the sake of generality and the purposes of this paper, in the sequel we will consider both $\nl$ and its associative extension $\mathsf{NL}^{\mathrm{as}}=\mathsf{NL}+\{\mathrm{\eqref{as1}},\mathrm{\eqref{as2}}\}$, where 
\[(\varphi\otimes\psi)\otimes\chi\ra\varphi\otimes(\psi\otimes\chi)\tag{AS1}\label{as1}\] and \[\varphi\otimes(\psi\otimes\chi)\ra(\varphi\otimes\psi)\otimes\chi.\tag{AS2}\label{as2}\] When dealing with $\mathsf{NL}$, chains of conjunctions of the form $\varphi\otimes\psi\otimes\chi$, (as e.g. the antecedent of (A7)) will be meant as associated to the left  (cf. ($\mathrm{A6}^{*}$)), but this choice has to be meant as purely conventional although the same is done in \textit{Principia Mathematica} (cf. \cite{Oreily,PrincipiaMath1}), which is among Nelson's main bibliographic references for the design of his system.

Until now, we have dealt with variants of $\mathsf{NL}$ whose inference rules apply \emph{only} to theorems. 
{However, in view of issues raised at the beginning of this section and} for the sake of completeness, we believe that two further systems based on stronger versions of Modus Ponens, adjunction, weak conjunction elimination, and replacement of equivalents might still be worth of consideration. Indeed, in this work we will be concerned also with logics obtained by replacing inference rules (MP), (Adj), (CE) and $\mathrm{(Eq)}_{1}$ of $\nl$ and $\nlas$ with the corresponding stronger inference rule schemes
\begin{center}
\AxiomC{$\varphi\ra\psi$}
\AxiomC{$\varphi$}\RightLabel{$\mathrm{(MP)}^{s}$}
\BinaryInfC{$\psi$}
\DisplayProof,
\AxiomC{$\varphi$}
\AxiomC{$\psi$}\RightLabel{$\mathrm{(Adj)}^{s}$}
\BinaryInfC{$\varphi\otimes\psi$}
\DisplayProof,
\AxiomC{$\varphi\Leftrightarrow\psi$}
\AxiomC{$\chi$}\RightLabel{$\mathrm{(Eq)}_{1}^{s}$}
\BinaryInfC{$\chi'$}
\DisplayProof,
\AxiomC{$\varphi\otimes\psi$}\RightLabel{$\mathrm{(CE)}^{s}$}
\UnaryInfC{$\varphi$}
\DisplayProof,
\end{center}
where $\chi'$ is obtained by replacing one or more occurrences of $\varphi$ by $\psi$ in $\chi$.   In what follows, we will refer to $\mathsf{NeL}=\langle\mathbf{Fm}_{\mathcal{L}},\vdash_{\mathsf{NeL}}\rangle$ as the derivability relation induced by the system obtained from $\nl$ upon replacing (MP), (Adj), (CE) and $\mathrm{(Eq)}_{1}$ with $\mathrm{(MP)}^{s}$, $\mathrm{(Adj)}^{s}$, $\mathrm{(CE)}^{s}$ and $\mathrm{(Eq)}_{1}^{s}$, respectively, and we set $\nelas=\nel+\{\mathrm{\eqref{as1}},\mathrm{\eqref{as2}}\}$.  

Making use of customary syntactic arguments, one can prove the inclusions provided by the next  remark summarizing some relationships between systems encountered so far.
\begin{remark}\label{rem:giuseppe}The following hold:
\begin{enumerate}
\item $\vdash_{\nl}\subseteq\vdash_{\nlas}\subseteq\vdash_{\nelas}$;
\item $\vdash_{\nl}\subseteq\vdash_{\nel}  \subseteq\vdash_{\nelas}$. 
\end{enumerate}
\end{remark}
\noindent Moreover, straightforward inductions on the length of proofs, yield the following result. 
\begin{proposition}\label{prop:odifreddi}For any $\varphi\in\mathrm{Fm}_{\mathcal{L}}$:
\begin{enumerate}
\item $\vdash_{\nl}\varphi$ iff $\vdash_{\nel}\varphi$;
\item $\vdash_{\nlas}\varphi$ iff $\vdash_{\nelas}\varphi$.
\end{enumerate}
 
\end{proposition}

\section{An algebraic-relational semantics for $\nel$}\label{sec:algsemnl} In this section we outline an algebraic-relational semantics for $\nel$, $\nelas$ and their extensions with finitely many rules having a finite number of premises. In view of Proposition \ref{prop:odifreddi}, we will obtain a (weak) completeness proof for $\nl$ and $\nlas$ as well.

 The background idea the semantics of systems outlined in the previous section relies on is that, while intensional conjunction $\otimes$, compatibility $\circ$, negation ${}^{*}$, and so $\ra$ have as a semantic counterpart algebraic (primitive or term-defined) operations over a set of propositions, in order to model semantically properties of entailment one has to introduce an \emph{external} relation of incompatibility between propositions (that still behaves adequately w.r.t. $\circ$), and between propositions and ``external'' truth values to the effect that ``being true'' can be rendered in terms of ``incompatibility with falsehood''.

We start from the concept of a weak $\mathcal{N}$-algebra. 

\begin{definition}A weak $\mathcal{N}$-algebra is an algebra $\alga=(A,\otimes,\circ,{}^{\ast})$ of type $(2,2,1)$ such that:
\begin{enumerate}
\item $(A,\otimes)$ and $(A,\circ)$ are commutative \emph{groupoids};
\item ${}^{\ast}:A\to A$ is an \emph{involution}, namely it satisfies \[x^{\ast\ast}\approx x.\]
\item The following hold:\[(x\otimes y)\circ z\approx (x\otimes z)\circ y.\tag{ex}\label{ex}\]
\end{enumerate}
Moreover, if $(A,\otimes)$ is also a \emph{semigroup} (i.e. $\otimes$ is associative), then $\alga$ will be simply called an $\mathcal{N}$-algebra.  
\end{definition}
\noindent Of course, weak $\N$-algebras ($\N$-algebras) form a variety henceforth denoted by $\mathcal{N}^{w}$ ($\mathcal{N}$).\\ 
Let us introduce a new pair $\{\mathsf{t},\mathsf{f}\}$ of symbols such that $\mathsf{t}\ne \mathsf{f}$. In the sequel, we assume that $\{\tr,\fa\}$ is always \emph{disjoint} from universes of algebras we will be concerned with (unless said otherwise). Moreover, for any $A\ne\emptyset$, we set $\overline{A}:=A\cup\{\mathsf{t},\mathsf{f}\}$\label{pippo}. $\tr$ and $\fa$ will stand for ``truth'' and ``falsehood'', respectively. 

Let $A$ be a non-empty set and let $R\subseteq A\times A$ be a binary relation over $A$. Using a customary notation, for any $a,b\in A$, $a R b$ will be short for $(a,b)\in R$.
\begin{definition}\label{def:nmod} A $\mathfrak{N}_{w}$-model is a pair $\mathcal{M}=(\alga,\perp,\{\tr,\fa\})$ such that $\alga$ is a weak $\mathcal{N}$-algebra and $\perp\ \subseteq \overline{A}\times\overline{A}$ is such that, for any $x,y,z\in A$,  the following hold:
\begin{enumerate}[(a)]
\item $x\perp x^{*}$;
\item $x\pp y^{*}$ and  $y\pp x^{*}$ {implies} $x=y$; 
\item $x\perp y$ {iff} $x\circ y\perp{\tr}$;
\item $x\perp{\tr}$ {iff} $x^{*}\perp{\fa}$;
\item $x\perp{\fa}$ and $y\perp{\fa}$ {iff} $x\otimes y\perp{\fa}$;
\item $(x\circ y^{*})^{*}\perp(x\circ y)^{*}$;
\item $x\perp y$ {and} $x\perp{\fa}$ {implies} $y\perp{\tr}$;
\item $x\nlra y\nlra z\perp((x\ra y)\ra((y\ra z)\ra(x\ra z)))^{*}$.
\end{enumerate}
\end{definition}
\noindent Whenever a $\mathfrak{N}_{w}$-model $\mathcal{M}=(\alga,\perp,\{\tr,\fa\})$ is such that $\alga\in\mathcal{N}$, it will be simply called an $\mathfrak{N}$-model. The class of $\mathfrak{N}_{w}$-models ($\mathfrak{N}$-models) will be henceforth denoted by $\mathfrak{N}_{w}$ ($\mathfrak{N}$). Given a (weak) $\mathfrak{N}$-model $\mathcal{M}=(\alga,\pp,\{\tr,\fa\})$, $\pp$ will be called an \emph{incompatibility relation} over $\alga$. 

\begin{example}\label{ex:1}Let us consider the triple $\mathcal{M}=(\alga=(\{a,b,c,d,e,f\},\otimes,\circ,{}^{*}),\pp,\{\tr,\fa\})$, such that $\alga=(\{a,b,c,d,e,f\},\otimes,\circ)$ is determined by the following Cayley tables
\begin{center}
\begin{tabular}{c| c c c c c c} $\circ$ & a & b & c & d & e & f\\ 
\hline  
a & a & e & e & a & e & a \\ 

b& e & a & e & a & e & a\\ 
c& e & e & e & a & e & e\\  
d& a & a & a & a & a & a\\  
e& e & e & e & a & c & a\\
f& a & a & e & a & a & a\\ 
\hline 
\end{tabular}\quad\quad
\begin{tabular}{c| c c c c c c} $\otimes$ & a & b & c & d & e & f\\ 
\hline  
a & c & b & c & d & c & b\\ 
b & b & d & b & d & b & d\\ 
c & c & b & c & d & c & b\\  
d & d & d & d & d & d & d\\  
e & c & b & c & d & c & b\\
f & b & d & b & d & b & d\\  
\hline 
\end{tabular}\quad\quad
\begin{tabular}{c| c } $\ast$ &   \\ 
\hline  
a& a \\ 
b& b \\ 
c& d \\  
d& c \\  
e& f \\
f& e \\  
\hline 
\end{tabular}   
\end{center}
and $\pp=\{(a,\fa),(c,\fa),(e,\fa),(a,\tr),(d,\tr),(f,\tr),(a,a),(a,d),(a,f),(b,b),(b,d),(b,f),(c,d),(d,a),\\ (d,b),(d,c),(d,e),(d,f),(e,d),(e,f),(f,a),(f,b),(f,d),(f,e),(f,f)\}$.
 Customary calculations yield that $\mathcal{M}$ is an $\mathfrak{N}$-model.  
\end{example}
 
\begin{remark}\label{rem:prop1}Let $\mathcal{M}=(\alga,\perp,\{\tr,\fa\})$ be a $\mathfrak{N}_{w}$-model. Then, for any $x,y\in A$: $x\perp y$ {implies} $x\circ y^{*}\perp{\fa}$. Indeed, assume that $x\perp y$. By Definition \ref{def:nmod}(c), one has $x\circ y\pp\tr$ and so, by 
Definition \ref{def:nmod}(e), $(x\circ y)^{*}\pp\fa$. By Definition \ref{def:nmod}(f),(g), one has $(x\circ y^{*})^{*}\pp\tr$ and so $x\circ y^{*}\pp\fa$. 
\end{remark}
It is interesting to point out that in any $\mathfrak{N}_{w}$-model $\mathcal{M}=(\alga,\pp,\{\tr,\fa\})$, one can define a notion of a true or false proposition upon considering $x\in A$ true if $x\pp\fa$, i.e. if $x$ is incompatible with falsehood, and  false if $x\pp\tr$, namely if $x$ is incompatible with truth. 

\noindent Given a $\mathfrak{N}_{w}$-model $\mathcal{M}=(\alga,\perp,\{\tr,\fa\})$, let $F^{\mathcal{M}}_{\pp}=\{a\in A:a\perp{\fa}\}$. Let us define $\models_{\mathfrak{N}_{w}}\ \subseteq\mathcal{P}(\mathrm{Fm}_{\mathcal{L}})\times\mathrm{Fm}_{\mathcal{L}}$ upon setting, for any $\Gamma\cup\{\varphi\}\subseteq\mathrm{Fm}_{\mathcal{L}}$,
$\Gamma\models_{\mathfrak{N}_{w}}\varphi$ iff there is $\Gamma'\subseteq\Gamma,\ |\Gamma'|<\omega$, such that, for any $\mathcal{M}=(\alga,\perp,\{\tr,\fa\})\in\mathfrak{N}_{w},$ and  $h\in\mathbf{Hom}(\mathbf{Fm}_{\mathcal{L}},\alga)$: $$h(\Gamma')\subseteq F^{\mathcal{M}}_{\pp}\text{ implies }h(\varphi)\in F^{\mathcal{M}}_{\pp}.$$ Whenever the reference to $\mathcal{M}$ will be clear from the context, we will write just $F_{\pp}$ in place of $F^{\mathcal{M}}_{\pp}$. 
Moreover, if $\mathcal{M}=(\alga,\perp,\{\tr,\fa\})\in\mathfrak{N}_{w}$,  $\mathcal{M}\models \varphi$ will be short for $(h(\varphi),\fa)\in\perp$, for any 
$h\in\mathbf{Hom}(\mathbf{Fm}_{\mathcal{L}},\alga)$. Furthermore, we set $\models_{\mathfrak{N}_{w}} \varphi$ provided that $\mathcal{M}\models\varphi$, for any $\mathcal{M}\in\mathfrak{N}_{w}$. 

\noindent The following remark is obvious.
\begin{remark}\label{rem:models}$\langle\mathbf{Fm}_{\mathcal{L}},\models_{\mathfrak{N}_{w}}\rangle$ is a sentential logic.
\end{remark}

The proof of the next theorem can be carried out by means of a slight modification of well known techniques. However, since the models we deal with are, in some sense, non-standard for algebraic logic, we will run a bit more into details.
\begin{theorem}\label{thm:compl}For any $\Gamma\cup\{\varphi\}\subseteq\mathrm{Fm}_{\mathcal{L}}$, the following holds: 
\[\Gamma\vdash_{\nel}\varphi\text{ iff }\Gamma\models_{\mathfrak{N}_{w}}\varphi.\]
\end{theorem}
\begin{proof}The left-to-right-direction can be proven by means of a customary induction on the length of proofs. Let $(\mathbf{A},\pp,\{\tr,\fa\})\in\mathfrak{N}_{w}$, $\varphi,\psi,\chi\in\mathrm{Fm}_{\mathcal{L}}$, and $h\in\mathbf{Hom}(\mathbf{Fm}_{\mathcal{L}},\alga)$ with $h(\varphi)=a,\ h(\psi)=b,\ h(\chi)=c$. Concerning (A1), one has that $a\perp a^{*}$ implies $a\circ a^{*}\perp\tr$ (Definition \ref{def:nmod}(c)) and so $(a\circ a^{*})^{*}\perp\fa$ (Definition \ref{def:nmod}(d)). (A2), (A3), \eqref{a5bis} and (A7) are straightforward and so they are left to the reader. As regards (A4), note that, by Definition \ref{def:nmod}(f), $(a\circ b^{*})^{*}\perp(a\circ b)^{*}$ which, in turn, entails $(a\ra b)\ra(a\circ b)=((a\circ b^{*})^{*}\circ(a\circ b)^{*})^{*}\pp\fa$. Concerning (A6), one has:
\begin{align*}
h(((\varphi\otimes\psi)\ra\chi)\ra((\varphi\otimes\chi^{*})\ra\psi^{*}))\pp\fa &\text{ iff }((a\otimes b)\ra c)\ra((a\otimes c^{*})\ra b^{*})\pp\fa\\
&\text{ iff }((a\otimes b)\circ c^{*})^{*}\ra((a\otimes c^{*})\ra b^{*})\pp\fa\\
&\text{ iff }((a\otimes c^{*})\circ b^{**})^{*}\ra((a\otimes c^{*})\ra b^{*})\pp\fa\\ 
&\text{ iff }((a\otimes c^{*})\ra b^{*})\ra((a\otimes c^{*})\ra b^{*})\pp\fa
\end{align*}
As the last condition obviously holds, the desired conclusion obtains.\\
As regards proper inference rules of $\nel$, we confine ourselves to consider $\mathrm{(Eq)}_{1}^{s}$ leaving cases involving $\mathrm{(MP)}^{s}$, $\mathrm{(Adj)}^{s}$ and \Ces\ to the reader. Suppose that $h(\chi)\pp\fa$ and  $h(\varphi\Leftrightarrow\psi)\pp\fa$. $h(\varphi\Leftrightarrow\psi)\pp\fa$ implies $(a\ra b)\otimes(b\ra a)=a\Leftrightarrow b\pp\fa $. So, by Definition \ref{def:nmod}(e) one has $(a\circ b^{*})^{*}=a\ra b\pp\fa$ and $(b\circ a^{*})^{*}=b\ra a\pp\fa$ and, by Definition \ref{def:nmod}(c),(d) also $a\pp b^{*}$ and $b\pp a^{*}$. Therefore $h(\chi')=h(\chi)\pp\fa$, for any $\chi'\in\mathrm{Fm}_{\mathcal{L}}$ obtained upon replacing one or more occurrences of $\varphi$ by $\psi$ in $\chi$.
\\
As regards the converse direction, we proceed through a customary Lindenbaum-Tarski type construction. Let $\Gamma\cup\{\varphi\}\subseteq\mathrm{Fm}_{\mathcal{L}}$. Suppose that $\Gamma\not\vdash_{\nel}\varphi$ and let $\mathrm{Cn}^{\nel}(\Gamma)$ be the theory generated by $\Gamma$ in $\nel$. For any $\psi,\chi\in\mathrm{Fm}_{\mathcal{L}}$, let
\[\psi\equiv_{\Gamma}\chi\text{ iff }\psi\ra\chi,\chi\ra\psi\in\mathrm{Cn}^{\nel}(\Gamma).\]
Let us show that $\equiv_{\Gamma}$ (henceforth $\equiv$, in brief) is a congruence over $\mathbf{Fm}_{\mathcal{L}}$. Reflexivity and symmetry are trivial. As regards  transitivity, suppose that $\psi_{1}\equiv\chi_{1}$ and $\chi_{1}\equiv\psi_{2}$, i.e. $\psi_{1}\ra\chi_{1},\chi_{1}\ra\psi_{1},\chi_{1}\ra\psi_{2},\psi_{2}\ra\chi_{1}\in\mathrm{Cn}^{\nel}(\Gamma)$. Note that $\psi_{1}\ra\chi_{1},\chi_{1}\ra\psi_{1}\vdash_{\nel}(\chi_{1}\ra\psi_{2})\ra(\psi_{1}\ra\psi_{2})$, by \Eqts,$\mathrm{(Adj)}^{s}$ and \Ces, so $(\chi_{1}\ra\psi_{2})\ra(\psi_{1}\ra\psi_{2})\in\mathrm{Cn}^{\nel}(\Gamma)$ and also $\psi_{1}\ra\psi_{2}\in\mathrm{Cn}^{\nel}(\Gamma)$, by \Mps. Similarly, one proves that $\psi_{2}\ra\psi_{1}\in\mathrm{Cn}^{\nel}(\Gamma)$. Consequently, one has $\psi_{1}\equiv\psi_{2}$. As regards the compatibility of $\equiv$ with respect to operations, one has e.g. that $\psi_{1}\ra\chi_{1},\chi_{1}\ra\psi_{1},\psi_{2}\ra\chi_{2},\chi_{2}\ra\psi_{2}\in\mathrm{Cn}^{\nel}(\Gamma)$ implies $\psi_{1}\otimes\psi_{2}\ra\psi_{1}\otimes\chi_{2},\psi_{1}\otimes\chi_{2}\ra\chi_{1}\otimes\chi_{2}\in\mathrm{Cn}^{\nel}(\Gamma)$, i.e. $\psi_{1}\otimes\psi_{2}\ra\chi_{1}\otimes\chi_{2}\in\mathrm{Cn}^{\nelw}(\Gamma)$, by several applications of \Eqts, \Mps, $\mathrm{(Adj)}^{s}$. The converse can be obtained analogously. So $\psi_{1}\otimes\psi_{2}\equiv\chi_{1}\otimes\chi_{2}$. Cases for the remaining connectives can be handled similarly, and so they are left to the reader.\\
We show that $(\mathbf{Fm}_{\mathcal{L}}/{\equiv},\otimes,\circ,{}^{*})$ is a weak $\mathcal{N}$-algebra. Showing that ${}^{*}$ is an involution can be easily done as in \cite{MarPa}, upon proving that $\vdash_{\nel}\psi^{**}\ra\psi$ holds. Proving that $(\mathbf{Fm}_{\mathcal{L}}/{\equiv},\otimes)$ is a commutative groupoid is almost obvious, as it follows from \eqref{a5bis}. That $(\mathbf{Fm}_{\mathcal{L}}/{\equiv},\circ)$ is a commutative groupoid as well is granted by (A2). Condition \eqref{ex} follows by (A6) and the fact that $\equiv$ is a congruence. Let $\perp_{\Gamma}\ \subseteq\ \overline{\mathrm{Fm}_\mathcal{L}/{\equiv}}\times\overline{\mathrm{Fm}_\mathcal{L}/{\equiv}}$ be defined upon setting, for any $\psi/{\equiv},\chi/{\equiv}\in\mathrm{Fm}_\mathcal{L}/{\equiv}$, $$\psi/{\equiv}\pp_{\Gamma}\chi/{\equiv}\text{ iff }\psi\ra\chi^{*}\in\mathrm{Cn}^{\nel}(\Gamma).$$ Moreover, we set $$\psi/{\equiv}\pp_{\Gamma}\fa\text{ iff }\psi\in\mathrm{Cn}^{\nel}(\Gamma),\text{ and }\psi/{\equiv}\pp_{\Gamma}\tr\text{ iff }\psi^{*}\in\mathrm{Cn}^{\nel}(\Gamma).$$ First, one has to prove that $\pp_{\Gamma}$ (henceforth denoted by $\pp$, for short) is independent from representatives. However, it can be easily shown this is a consequence of \Mps,\Eqts, $\mathrm{(Adj)}^{s}$ and \Ces. Next, we show that $\pp$ satisfies conditions (a)-(h) of Definition \ref{def:nmod}. As regards (a), note that $\psi\ra\psi^{**}$ is a theorem of $\nel$, and so one has $\psi/{\equiv}\pp\psi^{*}/{\equiv}$. (b), (e), (f), (g), (h) are obvious, while (d) is a consequence of $\varphi\equiv \varphi^{**}$. As regards (c), one has $\psi/{\equiv}\perp\chi/{\equiv}$ iff $(\psi\circ\chi^{**})^{*}=\psi\ra\chi^{*}\in\mathrm{Cn}^{\nel}(\Gamma)$ iff $(\psi\circ\chi)^{*}\in\mathrm{Cn}^{\nel}(\Gamma)$ iff   $\psi\circ\chi/{\equiv}\pp\tr$, since $(\psi\circ\chi^{**})^{*}\equiv (\psi\circ\chi)^{*}$. (h) is an easy consequence of (A7). Therefore, $\mathcal{M}=(\mathbf{Fm}_{\mathcal{L}}/{\equiv},\pp,\{\tr,\fa\})\in\mathfrak{N}_{w}$. Moreover, Let $h:\mathbf{Fm}_{\mathcal{L}}\to\mathbf{Fm}_{\mathcal{L}}/{\equiv}$ be the homomorphism obtained upon extending the (usual) mapping such that, for any $x\in Var$, $x\mapsto x/{\equiv}$, to the whole $\mathrm{Fm}_{\mathcal{L}}$. One clearly has that, for any $\Gamma'\subseteq\Gamma$, $|\Gamma'|<\omega$, $h(\Gamma')\subseteq F_{\pp}$ but $h(\varphi)\notin F_{\pp}$. In other words, $\Gamma\not\models_{\mathfrak{N}_{w}}\varphi$.
\end{proof}

In general, building on the proof strategy from Theorem \ref{thm:compl}, one can prove analogous results for \emph{finite} extensions of $\nel$. Let $\Theta\subseteq\mathcal{P}(\mathrm{Fm}_{\mathcal{L}})\times\mathrm{Fm}_{\mathcal{L}}$. We say that $\Theta$ is \emph{finite} provided that $|\Theta|<\omega$ and, for any $(\Gamma,\varphi)\in\Theta$, $|\Gamma|<\omega$. The extension $\nel+\Theta$ is said to be finite provided that $\Theta$ is finite. {Now, for any $\varphi\in\mathrm{Fm}_{\mathcal{L}}$, let us write $\varphi(x_{1},\dots,x_{n})$ to highlight that $\varphi$ contains all and only  propositional variables among  $x_{1},\dots,x_{n}$.  For any $(\Gamma,\varphi)\in\mathcal{P}(\mathrm{Fm}_{\mathcal{L}})\times\mathrm{Fm}_{\mathcal{L}}$ ($|\Gamma|<\omega$), let $(\Gamma,\varphi)^{\pp}$ be the condition  
\[\forall x_{1},\dots,x_{n}[\gamma_{1}(x_{1},\dots,x_{n})\pp\fa\text{ and }\dots\text{ and }\gamma_{m}(x_{1},\dots,x_{n})\pp\fa\text{ imply }\varphi(x_{1},\dots,x_{n})\pp\fa]\label{firstordcondition}
\]
if $\Gamma=\{\gamma_{1}(x_{1},\dots,x_{n}),\dots,\gamma_{m}(x_{1},\dots,x_{n})\}\ne\emptyset$, and \[\forall x_{1},\dots,x_{n}[\varphi(x_{1},\dots,x_{n})\pp\fa],\] otherwise.} Moreover, let $\mathfrak{N}_{w}^{\Theta}$ be the class of all and only the $\mathfrak{N}_{w}$-models satisfying $(\Gamma,\varphi)^{\pp}$, for any $(\Gamma,\varphi)\in\Theta$. Note that, for any $\Theta\subseteq\mathcal{P}(\mathrm{Fm}_{\mathcal{L}})\times\mathrm{Fm}_{\mathcal{L}}$, $\mathfrak{N}_{w}^{\Theta}$ is always non-empty as it contains at least a trivial $\mathfrak{N}_{w}$-model. Furthermore, let $\langle\mathbf{Fm}_{\mathcal{L}},\models_{\mathfrak{N}_{w}^{\Theta}}\rangle$ be defined as at p. \pageref{rem:models}. Upon making good use of the proof strategy from Theorem \ref{thm:compl}, one has the following result.
\begin{theorem}\label{thm:compl2}For any finite extension $\nel+\Theta$ and, for any $\Gamma\cup\{\varphi\}\subseteq\mathrm{Fm}_{\mathcal{L}}$:\[\Gamma\vdash_{\nel+\Theta}\varphi\text{ iff }\Gamma\models_{\mathfrak{N}_{w}^{\Theta}}\varphi.\]
\end{theorem}
Now, it can be seen that $\mathfrak{N}=\mathfrak{N}_{w}^{\{\eqref{as1},\eqref{as2}\}}$. Therefore, in view of Proposition \ref{prop:odifreddi}, one has also the following theorem.

\begin{theorem}\label{thm:compl3}For any $\Gamma\cup\{\varphi\}\subseteq\mathrm{Fm}_{\mathcal{L}}$, and for any $\Theta\subseteq\mathcal{P}(\mathrm{Fm}_{\mathcal{L}})\times\mathrm{Fm}_{\mathcal{L}}$:
\begin{enumerate}
\item $\vdash_{\nl}\varphi$ iff $\models_{\mathfrak{N}_{w}}\varphi$;
\item $\vdash_{\nlas}\varphi$ iff $\models_{\mathfrak{N}}\varphi$;
\item $\vdash_{\nl+\Theta}\varphi$ iff $\models_{\mathfrak{N}_{w}^{\Theta}}\varphi$;
\item $\vdash_{\nlas+\Theta}\varphi$ iff $\models_{\mathfrak{N}^{\Theta}}\varphi$. 
\end{enumerate}
\end{theorem}

It can be shown that, actually, $\mathfrak{N}_{w}$-models are into bijective  correspondence with reduced matrix models of $\nel$. 
Indeed, let $\alga=(A,\otimes,\circ,{}^{*})$ be an $\mathcal{L}$-algebra. 
We can trivially define a logical $\nel$-filter over $\alga$ to be a set $\emptyset\ne F\subseteq A$ such that, for any $h\in\mathbf{Hom}(\mathbf{Fm}_{\mathcal{L}},\alga)$, any axiom $\varphi$ of $\nel$, $\psi,\chi\in\mathrm{Fm}_{\mathcal{L}}$, 
and any term $\delta(x,\vec{z})\in\mathrm{Fm}_{\mathcal{L}}$, one has\footnote{See \cite[pp. 6]{Font16} for details about the notation.}
\begin{enumerate}
\item $h(\varphi)\in F$;
\item $h(\psi), h(\psi\ra\chi)\in F$ implies $h(\chi)\in F$;
\item $h(\varphi),h(\psi)\in F$ iff $h(\varphi\otimes\psi)\in F$;
\item $h(\psi\ra\chi),h(\chi\ra\psi)\in F$ implies $\delta^{\alga}(h(\psi),\vec{c})\ra\delta^{\alga}(h(\chi),\vec{c})\in F$, for any $\vec{c}\in \vec{A}$. 
\end{enumerate}
The class of  matrix models of $\nel$ coincides with the class of pairs $\langle\alga,F\rangle$, where $\alga=(A,\otimes,\circ,{}^{*})$ is an $\mathcal{L}$-algebra  and $F$ is a $\nel$-filter over $\alga$. Of course, $\nel$ is complete w.r.t. the class of all its matrix models, as this holds for any logic (cf. \cite[Theorem 4.16]{Font16}). Let $\langle\alga,F\rangle$ be a matrix model of $\nel$ and define $\equiv\ \subseteq A^{2}$ upon setting, for any $a,b\in A$:
\[a\equiv b\text{ iff }a\ra b,b\ra a\in F.\] Putting in good use clauses (1)-(4) defining logical filters of $\nel$, it is easily seen that $\equiv$ is actually a congruence relation. Moreover, items (1)-(4) ensure that, for any $a,b\in A$, $a\equiv b$ if and only if, for any $\delta(x,\vec{z})\in\mathrm{Fm}_{\mathcal{L}}$, and any $\vec{c}\in \vec{A}$, it holds that $$\delta^{\alga}(a,\vec{c})\in F\text{ iff }\delta^{\alga}(b,\vec{c})\in F.$$ Consequently, by \cite[Theorem 4.23]{Font16}, $\equiv\ =\ \Omega^{\alga}(F)$.\\
Given a matrix $\langle\alga,G\rangle$, where $\alga$ is an $\mathcal{L}$-algebra, and $G\subseteq A$, we define a relation $\pp_{G}\ \subseteq \overline{A}\times\overline{A}$, where $\overline{A}$ has to be meant as at p. \pageref{pippo}, upon setting
\[\pp_{F}:=\{(a,b)\in A^{2}:a\ra b^{*}\in G\}\cup\{(a,\fa):a\in G\}\cup\{(a,\tr):a^{*}\in G\}.\] Moreover, for any structure $(\alga,\pp,\{\tr,\fa\})$, where $\alga$ is an $\mathcal{L}$-algebra and $\pp\ \subseteq\overline{A}\times\overline{A}$, we set $F_{\pp}:=\{a\in A: a\pp\fa\}$.
\begin{proposition}\label{equivmatrixnelsonmo}For any pair $\langle\alga,G\rangle$, and any structure $(\alga,\pp,\{\tr,\fa\})$ such that $\alga$ is an $\mathcal{L}$-algebra, $\pp\ \subseteq \overline{A}\times\overline{A}$ and $G\subseteq A$, one has:
\begin{enumerate}
\item $\langle\alga, G\rangle$ is a reduced model of $\nel$ if and only if $(\alga,\pp_{G},\{\tr,\fa\})\in\mathfrak{N}_{w}$;
\item $(\alga,\pp,\{\tr,\fa\})\in\mathfrak{N}_{w}$ if and only if $\langle\alga,F_{\pp}\rangle$ is a reduced model of $\nel$ and $\pp=\pp_{F_{\pp}}$.
\end{enumerate}
Consequently, $\mathsf{Alg}^{*}\nel$ coincides with the class of algebraic reducts of $\mathfrak{N}$-models. Moreover, for any reduced matrix $\langle\alga,G\rangle$ of $\nel$ and any $\mathfrak{N}_{w}$-model $(\alga,\pp,\{\tr,\fa\})$, one has $$\langle\alga,G\rangle=\langle\alga,F_{\pp_{G}}\rangle\text{ and }(\alga,\pp,\{\tr,\fa\})=(\alga,\pp_{F_{\pp}},\{\tr,\fa\}).$$
\end{proposition}
\begin{proof}
(1). As regards the left-to-right direction, we first show that $\alga$ is a weak $\mathcal{N}$-algebra. That $(A,\otimes)$ is a commutative groupoid is ensured by noticing that, since $G$ is a logical filter and $\Omega^{\alga}(G)=\Delta$, for any $a,b\in A$, $a\otimes b\ra b\otimes a\in G$ and $b\otimes a\ra a\otimes b\in G$, so $a\otimes b\equiv b\otimes a$ and $a\otimes b = b\otimes a$. The commutativity of $\circ$, the involutivity of ${}^{*}$, as well as \eqref{ex}, can be proven exactly in the same way. To see that $\pp_{G}$ satisfies conditions from Definition \ref{def:nmod}, we confine ourselves to show (a) and (c) leaving the remaining clauses to the reader.  As regards (a), note that $a\ra a\in G$. This means that $a\ra a^{**}\in G$ and so $a\pp_{G}a^{*}$. Concerning (c), note that $a\pp_{G}b$ iff $a\ra b^{*}\in G$ iff $(a\circ b)^{*}\in G$ iff $a\circ b\pp_{G}\tr$.\\
Conversely, first note that $G=F_{\pp_{G}}$ as $x\in G$ iff $x\pp_{G}\fa$ iff $x\in F_{\pp_{G}}$. Moreover, $F_{\pp_{G}}$ is a logical filter by Theorem \ref{thm:compl} (soundness). The fact that $\Omega^{\alga}(G)=\Delta$ is ensured by Definition \ref{def:nmod}(b). \\
To prove (2), observe that if $(\alga,\pp,\{\tr,\fa\})\in\mathfrak{N}_{w}$, then $x\pp_{F_{\pp}}y$ iff $x\ra y^{*}\in F_{\pp}$ iff $x\ra y^{*}\pp\fa$ iff $x\circ y\pp\tr$ iff $x\pp y$. Moreover, $x\pp_{F_{\pp}}\fa$ iff $x\in F_{\pp}$ iff $x\pp\fa$ and, similarly, one has $x\pp_{F_{\pp}}\tr$ iff $x\pp\tr$. So $(\alga,\pp,\{\tr,\fa\})=(\alga,\pp_{F_{\pp}},\{\tr,\fa\})$. Consequently, we conclude that $\langle\alga,F_{\pp}\rangle$ is a reduced model by item (1). Conversely,  if $\langle\alga,F_{\pp}\rangle$ is a reduced model of $\mathsf{NeL}$, then (again by item (1)), $(\alga,\pp_{F_{\pp}},\{\tr,\fa\})\in\mathfrak{N}_{w}$ and the desired conclusion follows by $\pp_{F_{\pp}}=\pp$.\\
The moreover part is clear.
\end{proof}
\noindent In the light of Proposition \ref{equivmatrixnelsonmo}, we conclude that the map $\langle\alga,F\rangle\mapsto(\alga,\pp_{F},\{\tr,\fa\})$, where $\langle\alga,F\rangle$ is a reduced model of $\mathsf{NeL}$ is a bijection between the class of reduced models of $\nel$ and $\mathfrak{N}_{w}$. 

Due to \cite[Theorem 6.7]{Font16} and the presence of (A1) and \Mps, it follows that $\mathsf{NeL}$ is protoalgebraic (see \cite[Definition 6.1]{Font16}). This means, among other things, that $\mathsf{NeL}$ satisfies the \emph{Parameterised Local Deduction-Detachment Theorem} (\cite[Theorem 6.22]{Font16}). Therefore, we believe that the following problem, whose analysis is postponed to a future work, is worth of investigation.
\begin{question}Characterize a family of Deduction-Detachment sets (see \cite[Definition 6.21]{Font16}) for $\nel$.
\end{question}
Moreover, it can be seen that the set $\Delta(x,y):=\{x\ra y,y\ra x\}$ is a set of \emph{congruence formulas} for $\nel$ and $\nelas$. This is a consequence of \cite[Theorem 6.60]{Font16} and the proof of Theorem \ref{thm:compl} (completeness), since e.g. it can be seen that conditions
\begin{enumerate}
\item[(R)] $\vdash_{\nel}\Delta(x,x)$;
\item[(MoP)] $x,\Delta(x,y)\vdash_{\nel}y$;
\item[(Re)] $\Delta(x_{1},y_{1}), \Delta(x_{2},y_{2})\vdash_{\nel}\Delta(x_{1}\star x_{2}, y_{1}\star y_{2})$, for each $\star\in\{\circ,\otimes\}$, and $\Delta(x_{1},y_{1})\vdash_{\nel}\Delta(x_{1}^{*},y_{1}^{*})$.
\end{enumerate}
Therefore, $\nel$ and its extensions are \emph{equivalential} (see \cite[Definition 6.63]{Font16}).  
However, it naturally raises the question whether they are also \emph{algebraizable}.  
\begin{question}\label{prob:algebraiz} Prove or disprove that $\nel$ is algebraizable.  
\end{question}

\subsection{Remarks on (NSY) and hyperconnexivity}
Logics encountered so far have Aristotle's and Boethius' theses as theorems, namely for any $\mathsf{L}\in\{\nl,\nlas,\nel,\nelas\}$, and so for any finite extension thereof, one has $\vdash_{\mathsf{L}}(\varphi\ra\varphi^{*})^{*}$, $\vdash_{\mathsf{L}}(\varphi^{*}\ra\varphi)^{*}$ (Aristotle's Theses), and both $\vdash_{\mathsf{L}}(\varphi\ra\psi)\ra(\varphi\ra\psi^{*})^{*}$ and $\vdash_{\mathsf{L}}(\varphi\ra\psi^{*})\ra(\varphi\ra\psi)^{*}$ (Boethius' Theses). These are easy consequences of the definition of $\ra$, axioms (A3) and its converse (which can be derived as in \cite[pp. 419--420]{MarPa}), (A4), $\mathrm{(MP)}^{s}$ (or (MP)) and $\mathrm{(Eq)}_{1}^{s}$ (or $\mathrm{(Eq)}_{1}$).\\ In the light of these facts, it raises the question whether Nelson's logics outlined so far are also \emph{hyper-connexive}.\\ A logic $\mathsf{L}$ over a language containing a binary connective $\ra$ for implication, and a unary connective ${}^{*}$ for negation is said to be hyper-connexive provided that it is connexive and, moreover, the following hold, for any pair of formulas $\varphi,\psi$:
\[\mathrm{(HC1)}\ \vdash_{\mathsf{L}}(\varphi\ra\psi)^{*}\ra(\varphi\ra\psi^{*})\text{ and }\mathrm{(HC2)}\ \vdash_{\mathsf{L}}(\varphi\ra\psi^{*})^{*}\ra(\varphi\ra\psi).\]
There are well known connexive logics satisfying hyper-connexive theses, think e.g. to H. Wansing's connexive logic $\mathsf{C}$ (\cite{wansingstanford,Wans2}). However, some doubts concerning their plausibility have been raised, in particular when it comes to formalize natural language conditionals, see \cite[pp. 446--447]{mccall2012}.\\ 
Actually, if $\ra$ is meant according to Nelson's perspective on entailment, then the failure of hyper-connexive principles is arguably a natural \emph{desideratum}. Indeed, the validity of (HC1) or (equivalently)  of (HC2), would entail, for any $x,y\in Var$, the validity of \[(x \circ y)\ra(x\ra y),\] namely that if $x$ and $y$ are compatible, then $x\ra y$ holds. This is of course unacceptable as, for example, the proposition ``$4$ is even'' ($p$) and ``$3$ is prime'' ($q$) are compatible but ``If $4$ is even, then $3$ is prime'' need not hold, since $q$ is arguably not obtainable from $p$ by ``logical analysis'', at least in Nelson's sense (cf. \cite[p. 447]{Nelson1}). {A further motivation for a Nelsonian rejection of hyperconnexivity is provided by the following example. Let us consider the statements ($p$)``John has an even number of children'', ``John has more boys than girls''($q$). Since $p$ is compatible with \emph{both} $q$ and $q^{*}$, one would have that both $p\ra q$ and $p\ra q^{*}$ hold true. In other words, hyperconnexivity, together with axioms and inference rules Nelson's systems (think to (MP) and (A4)),  yields inconsistencies\footnote{We thank an anonymous referee for pointing this fact out to us.}. This fact is mirrored by the next proposition. It shows that any non-trivial $\mathfrak{N}_{w}$-model provides a counterexample to hyperconnexivity.} 

\begin{proposition}\label{prop:df}Let $\mathcal{M}=(\alga,\pp,\{\tr,\fa\})\in\mathfrak{N}_{w}$. Then, $\mathcal{M}\models (x\circ y)\ra x\ra y$ if and only if $\alga$ is trivial.
\end{proposition}
\begin{proof}
Note that $\mathcal{M}\models(x\circ y)\ra(x\ra y)$ if and only if $\alga\models x\circ y\approx x\ra y$. Now, if this is the case, one has 
\begin{align*}
\fa &\pp((x\ra x)\otimes(x\ra y))\ra((x\ra x)\otimes(x\ra y))\\
&=((x\ra x)\otimes(x\ra y))\circ((x\ra x)\otimes(x\ra y))\\
&=((x\ra y)\otimes((x\ra x)\otimes(x\ra y)))\circ(x\ra x)\\
&=(x\ra x)\ra((x\ra y)\otimes((x\ra x)\otimes(x\ra y))).  
\end{align*}
Since $x\ra x\pp\fa$, one has also $(x\ra y)\otimes((x\ra x)\otimes(x\ra y))\pp\fa$ and so $(x\ra y)=x\circ y=y\circ x=y\ra x\pp\fa$. We conclude that $x=y$.
\end{proof}
\noindent As a consequence, any $\mathfrak{N}_{w}$-model yields a counterexample to the symmetry of $\ra$.

\begin{proposition}\label{prop:df1}Let $\mathcal{M}=(\alga,\pp,\{\tr,\fa\})\in\mathfrak{N}_{w}$. Then, $\mathcal{M}\models (x\ra y)\ra(y\ra x)$ if and only if $\alga$ is trivial.
\end{proposition}
\begin{proof}
The statement follows by Proposition \ref{prop:df} upon noticing that $\mathcal{M}\models (x\ra y)\ra(y\ra x)$ entails that $\mathcal{M}\models (x\circ y)\ra(x\ra y)$. 
\end{proof}
\noindent Therefore, $\ra$, as axiomatized by $\nl,\nlas,\nel$ or $\nelas$ can be understood as a genuine full-fledged connexive implication. Even more, these logics can be somewhat regarded as \emph{strongly} non-symmetric, in the sense that each of their non-trivial models provide a counterexample to the symmetry of $\ra$. As a consequence, one has that there is no non-trivial superlogic of $\nel$ yielding a symmetric implication and which is complete in the sense of Theorem \ref{thm:compl} w.r.t. a subclass of $\mathfrak{N}_{w}$. This property is not shared e.g. by Connexive Heyting Logic $\mathsf{CHL}$ introduced in \cite{FaLePa}, and which turns out to be deductively equivalent to Intuitionistic Logic $\mathsf{IL}$. Indeed, $\mathsf{CHL}$ admits a largest non trivial superlogic $\mathsf{CBL}$ that is deductively equivalent to Classical Logic, and which yields a symmetric implication, cf. \cite[p. 113]{FaLePa}.

\section{Some extensions of Nelsonian logics}\label{sec:variantsNL}
\subsection{On strong connexivity  and the irreflexivity of $\pp$}\label{sec:strongsupconn} 
As we have recalled in the Introduction, one of the cornerstones of Nelson's perspective about logic is that consistency/compatibility is always reflexive, while inconsistency/incompatibility is always \emph{irreflexive}. Indeed, once meaning is taken into account, even contradictions are at least compatible with themselves.  However, Example \ref{ex:1} shows that axioms and inference rules of $\nel$ are not strong enough to ensure that, on the semantical side, $\mathfrak{N}_{w}$-models are endowed with an irreflexive incompatibility relation. 
 A further consequence of this fact is that there are $\mathfrak{N}_{w}$-models satisfying the formula $x\ra x^{*}$. So, $\nel$ is not \emph{strongly connexive} in the sense of A. Kapsner (see e.g. \cite{Kap1,Wans1}). Indeed, strong connexivity can be rendered by the following conditions (cf. \cite[p. 3]{Kap1}):
\begin{itemize}
\item[](Unsat1) In no (non-trivial) model, $\varphi\ra\varphi^{*}$ is satisfiable (for any $\varphi$);
\item[](Unsat2) In no (non-trivial) model $\varphi\ra\psi$ and $\varphi\ra\psi^{*}$ are satisfiable (for any $\varphi$ and $\psi$).
\end{itemize}
We observe that, in the present framework, (Unsat1) is equivalent to the irreflexivity of $\pp$. 
\begin{remark}Let $\nel^{\Theta}=\langle\mathbf{Fm}_{\mathcal{L}},\vdash_{\nel^{\Theta}}\rangle$ be a finite extension of $\nel$. Then $\nel^{\Theta}$ satisfies (Unsat1) if and only if, for any non-trivial $\mathcal{M}=(\alga,\pp,\{\tr,\fa\})\in\mathfrak{N}_{w}^{\Theta}$, $\pp$ is irreflexive. Indeed, $\nel^{\Theta}$ satisfies (Unsat1) if and only if, for any non-trivial $\mathcal{M}=(\alga,\pp,\{\tr,\fa\})\in\mathfrak{N}_{w}^{\Theta}$ and, for any $a\in A$, it never holds that $(a\circ a)^{*}\pp\fa$, i.e. that $a\pp a$.
\end{remark}
However, we have also something more. Indeed, we will see that the irreflexivity of $\pp$ has, as a syntactical counterpart, a weak form of \emph{superconnexivity} first introduced by A. Kapsner in \cite{Kap1}. To explain and motivate this concept, we refer to his words:
\begin{quote}
Now, one might wonder whether the criterion of strong connexivity can be expressed
in some manner in the object language itself, given that ARISTOTLE and BOETHIUS are
not up to this task.\\
A look at an analogous problem might bring us to some interesting ideas here. Classical logic and most non-classical logics validate the principle of explosion, $(A\land\neg A)\to B$. Let us remind ourselves of a philosophical argument that is sometimes given in favor of the validity of $(A\land\neg A)\to B$. This validity, it is claimed, is the most we can do to express, in the object language itself, the thought that a contradiction is unsatisfiable.
In analogy to this use of explosion to express the unsatisfiability of any contradiction, we might try to ask that $(A\to\neg A)\to B$ should be valid, in order to express in the object language that $(A\to\neg A)$ is unsatisfiable (and similarly for the rest of the connexive theses).
Call a logic that validates all of these schemata and satisfies all the requirements for strong connexivity superconnexive. \cite[pp. 3-4]{Kap1}
\end{quote}
Unfortunately, it turns out that, even under a \emph{modicum} of assumptions, like e.g. substitutivity and Modus Ponens (which are sound inference rules, in our setting), a superconnexive logic is \emph{trivial}. Nevertheless, if one relaxes the superconnexive axiom by substituting it with the weaker \emph{inference rule}
\[\varphi\ra\varphi^{*}\vdash\psi,\tag{I}\label{i}\]
then one may still hope of encoding  strong connexivity by means of superconnexive principles. And actually, in the framework we are dealing with, this is the case. 

To characterize the class of $\mathfrak{N}_{w}$-models whose non-trivial members are all and only those endowed with an irreflexive incompatibility relation it is sufficient to consider the class $\mathfrak{N}_{w}^{\mathrm{I}}$ of $\mathfrak{N}_{w}$-models  satisfying the condition
\[x\pp x\text{ implies }x\approx y.\tag{I}\] 
\begin{remark}\label{rem:giardino}For any $\mathfrak{N}_{w}$-model $\mathcal{M}=(\alga,\pp,\{\tr,\fa\})$, $\mathcal{M}$ satisfies \eqref{i} if and only if it satisfies the condition
 \[x\ra x^{*}\pp\fa\ \text{implies}\ y\pp\fa.\] Indeed, as regards the left-to-right direction, one has that $x\ra x^{*}\pp\fa$ implies $x\circ x\pp\tr$ and so $x\pp x$. Therefore, $\alga$ is trivial and one has $y\pp\fa$, for any $y\in A.$ Conversely, $x\pp x$ implies $x\ra x^{*}\pp\fa$. So $x\ra y\pp\fa$ and $y\ra x\pp\fa$ hold. This entails that $x\approx y$ must hold in $\alga$ and so $\alga$ is trivial.  
\end{remark}

\noindent Upon considering the extension $\nel+\{(I)\}$, one has the following
\begin{theorem}
For any $\Gamma\cup\{\varphi\}\subseteq\mathrm{Fm}_{\mathcal{L}}$, the following holds: 
\[\Gamma\vdash_{\nel+\{\mathrm{I}\}}\varphi\text{ iff }\Gamma\models_{\mathfrak{N}_{w}^{\mathrm{I}}}\varphi.\] 
\end{theorem}

\noindent Furthermore, to have a fully strongly connexive extension of $\nel$, and so satisfying also (Unsat2), one has to consider the logic obtained upon extending $\nel$ with the further inference rule schema 
\[\varphi\ra\psi^{*},\varphi\ra\psi\vdash\chi.\tag{I2}\label{i2}\] 
Indeed, take the class $\mathfrak{N}_{w}^{\mathrm{I+}}$ of $\mathfrak{N}_{w}$-models satisfying the condition
\[x\pp y^{*}\text{ and }x\pp y\quad\text{imply}\quad x\approx z.\tag{$\mathrm{I2}$}\label{i2}\]
\begin{remark}Upon reasoning as in Remark \ref{rem:giardino}, one has that \eqref{i2} is equivalent to the condition
\[x\ra y\pp\fa\text{ and } x\ra y^{*}\pp\fa\text{ imply } z\pp\fa.\]
 
\end{remark}
Clearly, $\mathfrak{N}_{w}^{\mathrm{I}}\subseteq\mathfrak{N}_{w}^{\mathrm{I+}}$ and $\mathfrak{N}_{w}^{\mathrm{I+}}$ is the class of $\mathfrak{N}_{w}$-models whose non-trivial members are all and only those not satisfying both $\varphi\ra\psi$, $\varphi\ra\psi^{*}$ (for any $\varphi$ and $\psi$), and so also $\varphi\ra\varphi^{*}$ (for any $\varphi\in\mathrm{Fm}_{\mathcal{L}}$).
\begin{theorem}
For any $\Gamma\cup\{\varphi\}\subseteq\mathrm{Fm}_{\mathcal{L}}$, the following holds: 
\[\Gamma\vdash_{\nel+\{\mathrm{I2}\}}\varphi\text{ iff }\Gamma\models_{\mathfrak{N}_{w}^{\mathrm{I+}}}\varphi.\] 
\end{theorem}
Considerations made so far give further substance to Kapsner's intuitions, as they show that concepts he introduces are entailed by, somewhat ``hidden'' below, ideas raised at the very dawn of modern connexive logic. However, as the same author points out:
\begin{quote}
[...]one might well have doubts concerning superconnexivity. A superconnexive logic requires an implication that always turns out to be true if the antecedent is unsatisfiable; many implicational connectives in the literature indeed behave this way, but there have also been deep reservations about those conditionals voiced by relevant and paraconsistent logicians. In the case at hand, they would balk at the validity of $(A\to\neg A)\to B$ because there is no connection between the antecedent and the consequent at all. Likewise, $(A\to\neg A)\models B$ would not meet what they would require of a decent consequence relation. \cite[pp. 4]{Kap1}.
\end{quote}
Actually, Kapsner's remark is relevant to our discourse since, as a consequence of previous results, in order to have a logic capable of  coping with Nelson's intuition about the irreflexivity of incompatibility, one has to endorse (at least one half of) Kapsner' strong connexivity and, in turn, the weakly superconnexive principle he recognizes as incompatible with a relevantist approach. And this may represent a challenge for Nelson's view of logic since, due to the way Nelson means its intensional nature, it is supposed to adhere to  kind of relevantist principles.\footnote{We highlight that we are not claiming that $\nel$ or any of the logics outlined in this paper are relevant logics in the narrow sense. Instead, as already remarked in the Introduction, we confine ourselves to point out that Nelson's conception of intensional logic is strongly tied to notions like e.g. ``real use in proof'' that most logicians accept as hallmarks of the relevantist  paradigm, see e.g. \cite{Mares}.} 
 However, we observe that in $\nel$ (as in all the logics we deal with in this work) two \emph{distinct} notions of entailment coexist. The first, embodied by $\vdash_{\nel}$, calls into question \emph{truth} (or incompatibility with falsehood). Such a consequence relation,  that we may call \emph{external}, establishes a relationships between sets of formulas and formulas to the effect that, whenever the former hold true in a model, the latter do. The second notion of ``consequence'', this time embodied by $\ra$, and that we may call \emph{internal}, is what Nelson really considers the true concept of intensional (relevant, as we might understand it) entailment. Therefore, if one confines oneself to considering the latter notion of entailment as the one whose ``relevant features'' should truly be taken into account (which might be reasonable, see Remark \ref{rem:explosionnoparad} and Remark \ref{prop:superconn} below), then assuming \emph{at least} $\nel+\{\mathrm{I}\}$ as a suitable extension of $\nel$,  capable of better aligning with Nelson’s intuition, might be worth considering..

In the light of the above remarks, if one aims at obtaining a full fledged strongly connexive logic, one may also consider extending $\nel$ with the further inference rule 
\[\varphi,\varphi^{*}\vdash\psi.\tag{ECQ}\label{encq}\]
The logic thus obtained, that we denote by $\nel^{\mathrm{ECQ}}$, is sound and complete w.r.t. the class $\mathfrak{N}_{w}^{S}\subseteq\mathfrak{N}_{w}$ of $\mathfrak{N}$-models $\mathcal{M}=(\alga,\pp,\{\tr,\fa\})$ satisfying the condition:
\[x\pp\tr\text{ and }x\pp\fa\text{ imply }z\pp\fa\tag{s}\label{str}.\]

\begin{example}\label{ex:efq} The structure $\mathcal{M}$ from Example \ref{ex.ceci} (see below) is a member of $\mathfrak{N}_{w}^{S}$. 
\end{example}

\begin{theorem} For any $\Gamma\cup\{\varphi\}\subseteq\mathrm{Fm}_{\mathcal{L}}$, the following holds: 
\[\Gamma\vdash_{\nel^{\mathrm{ECQ}}}\varphi\text{ iff }\Gamma\models_{\mathfrak{N}_{w}^{S}}\varphi.\]
\end{theorem}
\begin{remark}\label{rem:explosionnoparad}
It is worth observing that, even if we assume \emph{explosion}, the logic one obtains still avoids some paradoxes of material implication. For example, one has $\not\vdash_{\nel^{\mathrm{ECQ}}}x\ra(x^{*}\ra y)$. Indeed, upon considering the model provided by Example \ref{ex:efq}, one has that $a\ra(a^{*}\ra a)=(a\circ(b\circ b))^{*}=(a\circ a)^{*}=a^{*}=b\not\pp\fa$. The same counterexample shows e.g. that $\not\vdash_{\nel^{\mathrm{ECQ}}}x\ra(y\ra x)$. 
\end{remark}
\noindent Of course, the next further remark is in order. We leave its simple proof to the reader.
\begin{remark}\label{prop:superconn} \eqref{i2} (and so \eqref{i}) holds in $\nel^{\mathrm{ECQ}}$.
\end{remark}
\noindent Actually, although essential for the irreflexivity of $\pp$, we are not aware if Nelson considered (I), (I2) or EFQ as valid inference rules, at least whenever truth is in question. However, as it will be shown in the sequel, they are consequences of principles he accepts explicitly.

As already recalled in the Introduction, in several works Nelson assumes a neat position against the axiom of disjunction introduction
\[\varphi\ra\varphi\oplus\psi,\tag{D}\label{d}\]
where, {as it can be inferred from the use of De Morgan laws in some of his arguments} (see e.g. \cite[p. 269]{Nelson2}) $$\varphi\oplus\psi:=(\varphi^{*}\otimes\psi^{*})^{*}.$$
{We highlight  that, by $\oplus$, we are referring to the ``mere'' factual disjunction $\lor$ (using  Nelson's terms and notation, see \cite[p. 446]{Nelson1}, cf. \cite{Nelson2}) and not to Nelson's \emph{intensional logical sum} $\bigvee$ that he defines in terms of entailment upon setting $$\varphi\bigvee\psi:=\varphi^{*}\ra\psi,$$ and which is essentially different from $\lor$, although related to it. Indeed, as Nelson states in several venues (\cite[p. 132]{Nelson3}, cf. \cite[p. 446]{Nelson1}), if $p\bigvee q$ is true, then $p\oplus q$ must hold as well, although the converse is false, in general.\footnote{We observe that J. Woods in \cite[p. 65]{Woods1} interprets the relationship between Nelson's intensional disjunction and ``factual'' disjunction differently. Indeed, he  writes: ``$p$ or $q$ [\emph{intensional disjunction}], if true, is necessarily true, whereas, even if true, $p\lor q$ could be false''.}}\\
As Nelson observes (see e.g. \cite[pp. 273]{Nelson2}), there is no way of deducing $\varphi\oplus\psi$ from $\varphi$ by logical analysis since, although $\varphi\oplus\psi$ asserts less than $\varphi$ from the perspective of truth, once meaning is considered, what it asserts is \emph{more} than $\varphi$. 
Actually, beside its incompatibility with Nelson's philosophy, as a consequence of analogous results concerning conjunction simplification, \eqref{d} yields inconsistencies. 

\begin{proposition}$\nl+\{\eqref{d}\}$ is trivial. 
\end{proposition}
\begin{proof}
Note that, (a) $\vdash_{\nl}(\varphi\ra\psi)\ra(\psi^{*}\ra\varphi^{*})$. Consequently, $\vdash_{\nl+\{\eqref{d}\}}\varphi\ra\varphi\oplus\psi$ entails (by (a), (MP), and $\mathrm{(Eq)}_{1}$) $\vdash_{\nl+\{\eqref{d}\}}\varphi^{*}\otimes\psi^{*}\ra\varphi^{*}$ which, in turn, entails $\vdash_{\nl+\{\eqref{d}\}}\varphi\otimes\psi\ra\varphi$. Therefore, $\nl+\{\eqref{d}\}$ is trivial by Proposition \ref{prop:trivial}.
\end{proof}
\noindent Nevertheless, Nelson argues, like conjunction loses its ``unity'', namely its intensional character, when considered from the perspective of truth (and so outside the scope of entailment), to the effect that (CE) or $\mathrm{(CE)^{s}}$ become acceptable inference rules, the same holds for \emph{disjunction}. Indeed, as Nelson points out in \cite[p. 272]{Nelson2}, ``when we deal with truth-values we may disjoin to any true proposition any proposition and assert the whole thing. This is because the resulting disjunctive proposition ``$p$ or $q$'' is less determinate than the proposition with which we start'' (cf. \cite[p. 448]{Nelson1}). {Among other things, this passage provides a further evidence for concluding that what Nelson is referring to is not intensional disjunction as, otherwise, this would led us to admit that, if a proposition $\varphi$ is true, then $\varphi\bigvee\varphi=\varphi^{\star}\ra\varphi=(\varphi^{*}\circ\varphi^{*})^{*}$ is true. And this is, of course, unacceptable from Nelson's perspective.}

Therefore, let us consider the logics $\nl_{d}$ ($\nlas_{d}$) and $\nel_{d}$ ($\nelas_{d}$) obtained upon extending $\nl$ ($\nlas$) and $\nel$ ($\nelas$) with the inference rule schemas
\begin{center}
\AxiomC{$\vdash\varphi$}\RightLabel{$\mathrm{(DI)}$}
\UnaryInfC{$\vdash\varphi\oplus\psi$}
\DisplayProof and  
\AxiomC{$\varphi$}\RightLabel{$\mathrm{(DI)}^{s}$}
\UnaryInfC{$\varphi\oplus\psi$}
\DisplayProof,
\end{center}
respectively. Upon considering the class $\mathfrak{N}_{w}^{d}$ ($\mathfrak{N}^{d}$) of $\mathfrak{N}_{w}$-models ($\mathfrak{N}$-models) satisfying the further condition 
\begin{equation}
x\pp\fa\text{ implies }x\oplus y\pp\fa,\tag{di}\label{dis} 
\end{equation}
by Theorem \ref{thm:compl2} and Theorem \ref{thm:compl3}, one has the following
\begin{theorem}For any $\Gamma\cup\{\varphi\}\subseteq\mathrm{Fm}_{\mathcal{L}}$, the following hold:
\begin{enumerate}
\item $\Gamma\vdash_{\mathsf{NeL}+\{\mathrm{(DI)}^{s}\}}\varphi\text{ iff }\Gamma\models_{\mathfrak{N}_{w}^{\mathrm{d}}}\varphi$;
\item $\Gamma\vdash_{\nelas+\{\mathrm{(DI)}^{s}\}}\varphi\text{ iff }\Gamma\models_{\mathfrak{N}^{\mathrm{d}}}\varphi$;
\item $\vdash_{\nl+\{\mathrm{(DI)}\}}\varphi$ iff $\models_{\mathfrak{N}_{w}^{\mathrm{d}}}\varphi$;
\item $\Gamma\vdash_{\nlas+\{\mathrm{(DI)}\}}\varphi\text{ iff }\Gamma\models_{\mathfrak{N}^{\mathrm{d}}}\varphi$.
\end{enumerate}

\end{theorem}

\begin{example}\label{ex.ceci}
Let us consider the structure $\mathcal{M}=(\alga=(\{a,b,c,d,e,f\},\otimes,\circ,{}^{*}),\pp,\{\tr,\fa\})$ such that $\circ$, $\otimes$ and ${}^{*}$ are defined according to the following Cayley tables
\begin{center}
\begin{tabular}{c| c c c c c c } 
$\circ$ & a & b & c & d & e & f \\ 
\hline  
a& a & b & c & a & c & a \\ 
b& b & a & a & b & a & b \\
c& c & a & a & b & a & c \\
d& a & b & b & a & b & a \\
e&  c& a & a & b & a & b \\
f&  a&  b& c & a & b & a \\
\hline 
\end{tabular}\quad\quad
\begin{tabular}{c| c c c c c c } 
$\otimes$ & a & b & c & d & e & f\\ 
\hline  
a &  c & d & a & b & f & e \\ 
b&  d&  b& b& d& b& d\\
c&  a&  b& c& d& e& f\\
d&  b&  d& d& b& d& b\\
e&  f&  b& e& d& b& d\\
f&  e&  d& f& b& d& b\\
\hline 
\end{tabular}\quad\quad
\begin{tabular}{c| c } $\ast$ &  \\ 
\hline  
a &  b\\ 
b &  a\\
c &  d\\
d &  c\\
e &  f\\
f &  e\\
\hline 
\end{tabular}
\end{center}
while $\pp=\{(a,\fa),(c,\fa),(b,\tr),(d,\tr),(a,b),(b,a),(b,d),(b,f),(c,d),(d,b),(d,c),(d,e),(e,d),\\ (e,f),(f,b), (f,e)\}$. It can be verified that $\mathcal{M}\in\mathfrak{N}^{\mathrm{d}}$.
\end{example}
 
\begin{proposition}\label{prop:pizzy}The following hold: $$\varphi,\varphi^{*}\vdash_{\mathsf{NeL}+\{\mathrm{(DI)}^{s}\}}\psi.$$
\end{proposition}
\begin{proof}
Let $(\alga,\pp,\{\tr,\fa\})\in\mathfrak{N}_{w}^{d}$ and assume that $x\pp\fa$ and $x\pp\tr$, i.e. $x^{*}\pp\fa$. One has that $(x\otimes y)\pp(x\otimes y)^{*}$ entails that $(x\otimes y)\circ(x\otimes y)^{*}\pp\tr$ and so $(x\otimes(x\otimes y)^{*})\circ y\pp\tr$. Upon replacing $y$ by $((y\ra z)\otimes(z\ra y))^{*}$ in the previous expression, one has (a) $(x\otimes(x\otimes ((y\ra z)\otimes(z\ra y))^{*})^{*})\circ ((y\ra z)\otimes(z\ra y))^{*}\pp\tr$. Now, observe that $(x\otimes ((y\ra z)\otimes(z\ra y))^{*})^{*}=x^{*}\oplus((y\ra z)\otimes(z\ra y))$. Moreover, since $x^{*}\pp\fa$, one has that $x^{*}\oplus((y\ra z)\otimes(z\ra y))\pp\fa$. Since one has also $x\pp\fa$, we conclude $x\otimes(x\otimes ((y\ra z)\otimes(z\ra y))^{*})^{*}\pp\fa$. Finally, by (a), one has $x\otimes(x\otimes ((y\ra z)\otimes(z\ra y))^{*})^{*}\pp((y\ra z)\otimes(z\ra y))^{*}$ and so $((y\ra z)\otimes(z\ra y))^{*}\pp\tr$, i.e. $(y\ra z)\otimes(z\ra y)\pp\fa$. Consequently, $y=z$. Since $y$ and $z$ are arbitrary, we conclude that $\alga$ is trivial. So \eqref{str} holds. 
\end{proof}
As a consequence of the above considerations, it seems natural to assume as a possible ``comprehensive'' rendering of Nelson ideas, the logics $\nl+\{\mathrm{(DI)}\}$, $\nlas+\{\mathrm{(DI)}\}$, $\nel+\{\mathrm{(DI)}^{s}\}$ and $\nelas+\{\mathrm{(DI)}^{s}\}$, henceforth denoted by $\nl^{+},\nlas^{+},\nel^{+}$ and $\nelas^{+}$, respectively. In view of results from Section \ref{sec:strongsupconn}, $\nel^{+}$ and its variants are also strongly connexive, (weakly) superconnexive, and, in view of Proposition \ref{prop:df} and Proposition \ref{prop:df1}, in some sense \emph{strongly non-symmetric} and \emph{strongly non-hyperconnexive}.

We close this section with the following remark concerning the idempotency of $\otimes$. As inferred by \cite{MarPa}, Nelson regards conjunction as idempotent, since as he states in \cite{Nelson2}, ``$p\land p$ must be given no significance, either in operations or in inferences, beyond that possessed by $p$'', cf. \cite[p. 411]{MarPa}. Therefore, as the reader might wonder, adding the following pair of axioms might be still compatible with Nelson's perspective:
\[\mathrm{(Id1)}\ \varphi\otimes\varphi\ra\varphi\quad\quad \mathrm{(Id2)}\ \varphi\ra\varphi\otimes\varphi.\]
However, as we observe below, the extension of $\nl^{+}$ (and so of its superlogics) by means of these two axioms is far from being harmless. Indeed, let us consider the logic $\nl^{+}+\{\mathrm{(Id1)},\mathrm{(Id2)}\}$.  One has the following 
\begin{proposition}$\nl^{+}+\{\mathrm{(Id1)}\}$ is trivial.
\end{proposition}
\begin{proof}We proceed through a simple syntactical argument.
\begin{align*}
\mathrm{(1)} & \varphi\otimes\varphi\ra\varphi & \mathrm{(Id1)}\\
\mathrm{(2)} & (\varphi\otimes\varphi\ra\varphi)\ra(\varphi\otimes\varphi^{*}\ra\varphi^{*}) & \mathrm{(A6)}\\
\mathrm{(3)} & \varphi\otimes\varphi^{*}\ra\varphi^{*} & \mathrm{(MP)}
\end{align*}
So, one has $\vdash_{\nl^{+}}\varphi\otimes\varphi^{*}\ra\varphi^{*}$. Similarly, one proves $\vdash_{\nl^{+}}\varphi\otimes\varphi^{*}\ra\varphi$. Therefore, our conclusion follows from Proposition \ref{prop:pizzy} upon noticing that $\nl^{+}$ has the same theorems of $\nel^{+}$.
\end{proof}
The impossibility of assuming an idempotent conjunction on pain of ``harmful'' consequences is witnessed, whenever $\nlas$ (and so $\nelas$ or $\nelas^{+}$) is considered, also by the following fact.
\begin{proposition}For any $\varphi,\psi\in\mathrm{Fm}_{\mathcal{L}}$: 
\begin{equation}
\vdash_{\nlas}(\varphi\otimes\varphi)\circ(\psi\otimes\psi). \label{form:drawback}
\end{equation}
\end{proposition}
\begin{proof}
Indeed, for any $\mathfrak{N}$-model $\mathcal{M}=(\alga,\pp,\{\fa,\tr\})$, and for any $a,b\in A$, one has $(a\otimes a)\circ(b\otimes b)=(a\otimes(b\otimes b))\circ a=((a\otimes b)\otimes b)\circ a=(a\otimes b)\circ(a\otimes b)\pp\fa$. 
\end{proof}
\noindent Unlike other theses provable in Nelsonian logics, the logical meaning of \eqref{form:drawback} is less clear. For example, although one has that $\not\vdash_{\nlas}\varphi\circ\varphi^{*}$, for any $\varphi\in\mathrm{Fm}_{\mathcal{L}}$, \eqref{form:drawback} yields that $\vdash_{\nlas}(\varphi\otimes\varphi)\circ(\varphi^{*}\otimes \varphi^{*})$. In fact, there is no evidence as to whether Nelson was indeed aware of this fact nor, if this is the case, if it was among the reasons why he never took a clear position in favor of the associativity of $\otimes$. However, this is a matter of historical investigation we will not go through further. 

\section{A further Nelsonian logic (?)}\label{sec:n1}
As recalled in the Introduction, Nelson rejects (see e.g. \cite[p. 448]{Nelson1}) what he calls the \emph{Principle of the Syllogism} \[((\varphi\ra\psi)\otimes(\psi\ra\chi))\ra(\chi\ra\psi)\tag{Syl}\label{sil}\] stating the transitivity of entailment. Indeed, assuming it would yield, together with other inference rules of $\nl$, instances of \eqref{simpl} that conflict with the Nelsonian idea that conjunctions, whenever taken in relationship with entailment, must function as a unit, and so ``producing'' something more than the mere aggregate of meanings of their arguments. Think e.g. to the formula \[(\varphi\ra\varphi)\otimes(\varphi\ra\psi)\ra(\varphi\ra\psi).\]
Regardless doubts concerning the tenability of Nelson's arguments against \eqref{sil} raised in the literature by D. J. Bronstein (see  \cite{Bronstein}, cf. \cite{MarPa,Nelson4}), we argue that some passages taken from Nelson's work might hint that some form of transitivity could still be considered as compatible with his intuitions. In \cite[pp. 447]{Nelson1}, one has ``entailment is an identity between a structural part (though not necessarily less than the whole) of the antecedent and the entire consequent. That is to say, an entailment represents a logical analysis''. Of course, as Nelson clarifies in \cite[pp. 272]{Nelson2}, such a relationship between the antecedent and the consequent of a sound entailment should not be read in ``spatial'' terms, but rather in the same sense that we mean whenever we say that the conclusion of a sound syllogism is contained into its premises. Now, if we \emph{assume} that $\varphi\ra\psi$ \emph{and} $\psi\ra\chi$, then $\varphi\ra\chi$ might not follows from premises due to reasons already recalled above. But if $\varphi\ra\psi$ and $\psi\ra\chi$ are \emph{true}, to the effect that the meaning of $\psi$ is contained into the meaning of $\varphi$ and the meaning of $\chi$ is contained into the meaning of $\psi$, then the meaning of $\chi$ must be arguably contained into the meaning of $\varphi$ as well. Therefore, we believe that the extension of $\nel^{+}$ ($\nl^{+},\nlas^{+}, \nelas^{+}$) by means of the  inference rule   
\[\varphi\ra\psi,\psi\ra\chi\vdash\varphi\ra\chi.\tag{$\mathrm{A8}$}\label{a8}\]
might still be worth of consideration from a Nelsonian perspective. 

In the sequel we will deal with the system $\mathsf{NeL}^{1}$ ($\nelas^{1}$) obtained upon adding (A8) to $\nel^{+}$ ($\nelas^{+}$). Beside being arguably still compatible with a Nelsonian perspective, the interest in such a system is motivated also by the fact that $\mathsf{NeL}^{1}$ has a relational-algebraic semantics with a rather smooth behavior, since, for any $\mathcal{M}=(\alga,\pp,\{\tr,\fa\})\in\mathfrak{N}_{w}^{d}$ ($\in\mathfrak{N}^{d}$) of $\nel^{1}$ ($\nelas^{1}$), $\pp$ now induces a \emph{partial order} over the universe $A$ of $\alga$. As it will be seen, this allows us to consider some properties of models which make it possible to establish a relationship between models of $\nel^{1}$ ($\nelas^{1}$) and structures arising from a seemingly far apart context, orthoposets (ortholattices).   

We start from the definition of the semantics of $\nel^{1}$ ($\nelas^{1}$). Let $\mathfrak{N}_{w}^{1}$ and $\mathfrak{N}^{1}$ be the classes of $\mathfrak{N}_{w}^{d}$-models and $\mathfrak{N}^{d}$-models, respectively, satisfying the further condition
\begin{equation}
x\pp y^{*} \text{ and } y\pp z^{*} \text{ imply } x\pp z^{*}  
\end{equation}
or, equivalently,
\[x\ra y\pp\fa\text{ and }y\ra z\pp\fa\text{ imply }x\ra z\pp\fa.\]
\begin{remark}It is worth noticing that, for any $\mathcal{M}=(\alga,\pp,\{\tr,\fa\})\in\mathfrak{N}_{w}$, $\mathcal{M}\in\mathfrak{N}_{w}^{1}$ iff, for any $a,b\in A$, it holds that $$a\pp b\text{ iff there is }c\in A:\ a\pp c\text{ and }b\pp c^{*}.$$ Therefore $\mathfrak{N}_{w}^{1}$ coincides with the class of models in which any two propositions $a,b$ are incompatible if and only they have as a part of their meaning contradictory propositions.
\end{remark}
\noindent By Theorem \ref{thm:compl2} we have the following result. 
\begin{theorem}\label{thm:compl1}For any $\Gamma\cup\{\varphi\}\subseteq\mathrm{Fm}_{\mathcal{L}}$, the following hold:
\begin{enumerate}
\item $\Gamma\vdash_{\nel^{1}}\varphi\text{ iff }\Gamma\models_{\mathfrak{N}_{w}^{1}}\varphi$;
\item $\Gamma\vdash_{\nelas^{1}}\varphi\text{ iff }\Gamma\models_{\mathfrak{N}^{1}}\varphi$; 
\end{enumerate}
\end{theorem}
\noindent It is important to observe that, adding (A8) still allows us to renounce to the stronger ``non-Nelsonian'' axiom 
\begin{equation}
((\varphi\ra\psi)\otimes(\psi\ra\chi))\ra(\varphi\ra\chi). \label{eq:nonnel}
\end{equation}
\begin{example}\label{ex.pia}The structure from Example \ref{ex.ceci} is also an $\mathfrak{N}_{w}^{1}$-model. Moreover, it can be seen that the homomorphism $h$ obtained upon extending the mapping assigning to distinct variables $x,y,z$ the element $a$, one has $h(((x\ra y)\otimes(y\ra z))\ra(x\ra z)) = (a\ra a)\otimes(a\ra a)\ra(a\ra a) = a\otimes a\ra a = c\ra a = b\not\pp\fa$.
\end{example}
\noindent Example \ref{ex.pia} naturally raises the question as to whether (A8) can be derived within $\nel^{+}$ or $\nelas^{+}$. Actually, one has the following remark. We call $(\alga,\pp,\{\tr,\fa\})\in\mathfrak{N}_{w}$ \emph{regular} provided that, for any $x,y\in A$, $x\ne y$ iff $x\not\Leftrightarrow y\pp\fa$.

\begin{remark}
Let $\mathcal{M}=(\alga,\pp,\{\tr,\fa\})\in\mathfrak{N}_{w}^{d}$ be regular. Then $\mathcal{M}\in\mathfrak{N}_{w}^{1}$. Indeed, suppose that $x\pp y^{*}$ and $y\pp z$. If $x = y$ or $y = z^{*}$ or $x=z^{*}$, then there is nothing to prove. If $x\neq y$, $y\ne z^{*}$, and $x\ne z^{*}$ then $x\not\Leftrightarrow y\pp\fa$,  $y\not\Leftrightarrow z^{*}\pp\fa$ and $x\not\Leftrightarrow z^{*}\pp\fa$. So we conclude that $((x\not\Leftrightarrow y)\otimes(y\not\Leftrightarrow z^{*}))\otimes(x\not\Leftrightarrow z^{*})\pp\fa$ which, in turn, entails that $(x\ra y)\ra((y\ra z^{*})\ra(x\ra z^{*}))\pp\fa$. The desired conclusion follows upon applying Definition \ref{def:nmod}(g). 
\end{remark}
In general, if we call \emph{full} a $\mathfrak{N}_{w}$-model $(\alga,\pp,\{\tr,\fa\})\in\mathfrak{N}_{w}^{S}$ provided that for any $x\in A$, either $x\pp\fa$ or $x\pp\tr$ one has the following easy but meaningful remark.
\begin{proposition}The following are equivalent:
\begin{enumerate}
\item $(\alga,\pp,\{\tr,\fa\})$ is a full $\mathfrak{N}_{w}$-model;
\item $(\alga,\pp,\{\tr,\fa\})$ is a full $\mathfrak{N}_{w}^{d}$-model;
\item $(\alga,\pp,\{\tr,\fa\})$ is a full $\mathfrak{N}_{w}^{1}$-model.
\end{enumerate}
\end{proposition}
\begin{proof}(1) implies (2). Note that if $x\pp\fa$, then one must have $x\oplus y\pp\fa$, otherwise $x\oplus y\pp\tr$ entails that $x^{*}\otimes y^{*}\pp\fa$ and so $x^{*}\pp\fa$, i.e. $x\pp\tr$ which is impossible. (2) entails (3), since it is easily seen that any full $\mathcal{N}_{w}^{d}$-model is also regular. That (3) implies (1) is clear.
\end{proof}
Note that, unfortunately, not any $\mathfrak{N}_{w}^{1}$-model is full. The structure from Example \ref{ex.ceci} is a case in point. Actually, $\mathfrak{N}_{w}^{1}$-models need not be even regular, as witnessed by the following example. 
\begin{example}Let us consider the following structure $(\alga=(\{a,b,c,d,e\},\circ,\otimes,{}^{*}),\pp,\{\tr,\fa\})$ such that $\alga$'s operations are specified according to the following Cayley tables
\begin{center}
\begin{tabular}{c| c c c c c c } 
$\circ$ & a & b & c & d & e & f \\ 
\hline  
a& a & a & c & d & d & a \\ 
b& a & a & c & c & d & a \\
c& c & c & a & a & a & c \\
d& d & c & a & a & a & c \\
e& d & d & a & a & a & c \\
f& a & a & c & c & c & a \\
\hline 
\end{tabular}\quad\quad
\begin{tabular}{c| c c c c c c } 
$\otimes$ & a & b & c & d & e & f\\ 
\hline  
a&  e &  d & f& b& a& c\\ 
b&  d &  c & f& f& b& c\\
c&  f &  f & c& c& c& f\\
d&  b &  f & c& c& d& f\\
e&  a &  b & c& d& e& f\\
f&  c &  c & f& f& f& c\\
\hline 
\end{tabular}\quad\quad
\begin{tabular}{c| c } $\ast$ &  \\ 
\hline  
a &  c\\ 
b &  d\\
c &  a\\
d &  b\\
e &  f\\
f &  e\\
\hline 
\end{tabular}
\end{center}
while $\pp=\{(a,\fa),(e,\fa),(c,\tr),(f,\tr),(a,c),(b,c),(b,d),(c,a),(c,b),(c,f),(d,b),(d,f),(e,f),\\ (f,c),(f,d),(f,e)\}$. It can be seen that $\mathcal{M}$ is an $\mathfrak{N}_{w}^{1}$-model. However, one has that $a\ne b$ but $a\not\Leftrightarrow b=((a\ra b)\otimes(b\ra a))^{*}=(b\otimes a)^{*}=d^{*}=b\not\pp\fa$.
\end{example}

However, we believe the following problem is worth of some attention
\begin{question}
Characterize the least extension of $\nel^{+}$ which is sound and complete w.r.t. a class of non trivial full $\mathfrak{N}_{w}^{1}$-models. 
\end{question}
\noindent Actually, we have not yet succeeded to prove or disprove that (A8) is actually  derivable in $\nel^{+}$ or in $\nelas^{+}$. Therefore, we leave this as an open problem whose in-depth analysis is postponed to  future investigation. 
\begin{question}Prove or disprove that (A8) is derivable in $\nel^{+}$ or in $\nelas^{+}$.
\end{question}
In the sequel we show that $\mathfrak{N}_{w}^{1}$-models yield partially ordered sets endowed with a unary operation which are well known in the realm of an apparently far apart research field: quantum logic.  
  
\begin{definition}\label{def:bpai}A \emph{poset with antitone involution} is a first order structure $\mathbf{P}=(P,\leq,{}^{*})$ such that $(P,\leq)$ is a poset, and ${}^{*}: P\to P$ is an antitone involution, i.e. it is a mapping satisfying the following conditions:
\[x^{**}\approx x,\text{ and }x\leq y\ \text{implies}\ y^{*}\leq x^{*}.\]
Whenever $x\leq y^{*}$, $x$ and $y$ are said to be \emph{orthogonal} ($xOy$, in brief).
\end{definition}
\noindent A poset with antitone involution $\mathbf{P}=(P,\leq,{}^{*})$ is said to be \emph{bounded} provided it has a least and a greatest element. As customary, whenever a poset with antitone involution $\mathbf{P}=(P,\leq,{}^{*})$ is bounded, we will denote its least and greatest elements by $0$ and $1$, respectively, and include them in the signature of $\mathbf{P}$ as constants, i.e. we set $\mathbf{P}=(P,\leq,{}^{*},0,1)$. Moreover, if $\mathbf{P}$ is lattice-ordered, we will regard it as an algebra $(P,\land,\lor,{}^{*},0,1)$ of type $(2,2,1,0,0)$ where $\land$ and $\lor$ denote lattice operations.

Let $\mathbf{P}=(P,\leq,{}^{*})$ and $\mathbf{Q}=(Q,\leq,{}^{*})$ be posets with antitone involution. $\mathbf{Q}$ is said to be a \emph{subposet with antitone involution} of $\mathbf{P}$ provided that $Q\subseteq P$ and $x^{{*}_{\mathbf{Q}}}=x^{{*}_{\mathbf{P}}}$, for any $x\in Q$. An \emph{orthohomomorphism} from $\algp$ to $\mathbf{Q}$ is a map $h:P\to Q$ such that, for any $x,y\in P$, it holds that $x\leq y$ implies $h(x)\leq h(y)$ and $h(x^{*})=h(x)^{*}$. If, moreover, $h$ is also injective (injective and surjective) then it will be called an orthoembedding (orthoisomorphism).

Let $\mathbf{P}=(P,\leq)$ be a poset. Using a customary notation, for any $X\subseteq P$, we set
\[L(X):=\{y:\ y\leq x,\text{ for any }x\in X\}\text{ and }U(X):=\{y:x\leq y,\text{ for any }x\in X\}.\]
Whenever $X=\{x_{1},\dots,x_{n}\}$, for some $n\geq 1$, we will write $L(x_{1},\dots,x_{n})$ and $U(x_{1},\dots,x_{n})$ in place of $L(\{x_{1},\dots,x_{n}\})$ and $U(\{x_{1},\dots,x_{n}\})$, respectively.

The next definition provides a non-necessarily lattice-ordered generalization of the concept of an \emph{ortholattice} introduced for the first time by G. Birkhoff and J. von Neumann in their algebraic inquiries into the logico-algebraic foundation of Quantum Theory, see \cite{BV36}.
\begin{definition}A bounded poset with antitone involution $\mathbf{P}=(P,\leq,{}^{*},0,1)$ is said to be an \emph{orthoposet} provided that the following condition holds:
\[L(x,x^{*}) =\{0\}.\]
In this case, ${}^{*}$ is called an \emph{orthocomplementation}. Moreover, if $\leq$ is a lattice-order, then $\algp$ will be called an \emph{ortholattice}. 
\end{definition}
The next remark shows that any $(\alga,\pp,\{\tr,\fa\})\in\mathfrak{N}_{w}^{1}$ yields a poset with antitone involution to the effect that $\pp$ can be regarded as an orthogonality relation. 
\begin{remark}\label{rem:giambattista}
If $(\alga,\pp,\{\tr,\fa\})\in\mathfrak{N}_{w}^{1}$, then upon defining $\precsim\ \subseteq A^{2}$ by setting, for any $x,y\in A$,  
\[x\precsim y\text{ iff }x\pp y^{*},\] one obtains that the structure $\mathcal{I}^{\mathcal{M}}(\alga)=(A,\precsim,{}^{*})$ is a poset with antitone involution. Indeed, by (i), $\precsim$ is transitive. Antisymmetry follows by Definition  \ref{def:nmod}(b), while antitonicity of ${}^{*}$ is a direct consequence of the symmetry of $\pp$, as $x\precsim y$ iff  $x\pp y^{*}$ iff $x\circ y^{*}\pp\tr$ iff $y^{*}\circ x\pp\tr$ iff $y^{*}\pp x=x^{**}$ iff $y^{*}\precsim x^{*}$. Moreover, one obviously has that \[x\pp y\text{ iff }x\precsim y^{*}.\] We remark that $\mathcal{I}^{\mathcal{M}}(\alga)$ need not be a lattice, nor bounded. Indeed, upon considering the structure $\mathcal{M}=(\alga=(\{a,b,c,d,e,f\},\otimes,\circ,{}^{*}),\pp,\{\tr,\fa\})$ from Example \ref{ex.ceci}, one has that $\mathcal{I}^{\mathcal{M}}(\alga)$ has the shape depicted in Figure \ref{fig:benzenel}.
\begin{figure}
\centerline{\xymatrix{
a = b^{*} \ar@{-}[d]& & c=d^{*}\ar@{-}[d]\\
f=e^{*} \ar@{-}[d]& & e=f^{*}\ar@{-}[d]\\
d=c^{*} & & b=a^{*}
}}
\caption{The poset $\mathcal{I}^{\mathcal{M}}(\alga)$ from Example \ref{ex.ceci}}
\label{fig:benzenel}
\end{figure}
\end{remark}
\noindent Whenever the reference to the model will be clear from the sequel, we will simply write $\mathcal{I}(\alga)$ in place of $\mathcal{I}^{\mathcal{M}}(\alga)$.

Next, in order to speed up computations, we show that, for any $(\alga,\pp,\{\tr,\fa\})\in\mathfrak{N}_{w}^{1}$, once $\alga$ is ordered by $\precsim$, it yields a partially ordered commutative involutive residuated groupoid. Cf. e.g. \cite{Ch3,MetPaTsi}.
\begin{definition}\emph{A partially ordered commutative involutive  residuated groupoid} is a partially ordered algebra $\alga=(A,\leq,\supset,\otimes,{}^{*})$ of type $(2,2,1)$ such that:
\begin{enumerate}
\item $(A,\otimes)$ is a commutative groupoid;
\item $(A,\leq,{}^{*})$ is a poset with antitone involution;
\item For any $x,y,z\in A$:
\[x\otimes y\leq z\text{ iff }x\leq y\supset z.\]
\end{enumerate}
\end{definition}
\begin{remark}For any partially ordered commutative involutive residuated groupoid $\alga=(A,\leq,\supset,\otimes,{}^{*})$, one has that $x\leq y$ implies $x\otimes z\leq y\otimes z$ and $x\oplus z\leq y\oplus z$.
\end{remark}
For any $(\alga,\pp,\{\tr,\fa\})\in\mathfrak{N}_{w}^{1}$, we set 
\[x\supset y:=x^{*}\oplus y.\]
Clearly, $x\supset y$ might be regarded as a kind of material implication. 
\begin{proposition}\label{lem:partordsemgroup}Let $(\alga,\pp,\{\tr,\fa\})\in\mathfrak{N}_{w}^{1}$. For any $x,y,z\in A$:
\begin{equation}
x\otimes y\precsim z\text{ iff }x\precsim y\supset z.
\end{equation}
Consequently, the structure $(A,\precsim,\otimes,\supset,{}^{*})$ is an involutive commutative partially ordered groupoid.
\end{proposition}
\begin{proof}
Note that, for any $x,y,z\in A$, one has $x\otimes y\precsim z$ iff $x\otimes y\pp z^{*}$ iff $(x\otimes y)\circ z^{*}\pp\tr$ iff $(z^{*}\otimes y)\circ x\pp\tr$ iff $x\pp(z^{*}\otimes y)=(z^{*}\otimes y^{**})^{**}=(y^{*}\oplus z)^{*}=(y\supset z)^{*}$ iff $x\precsim y\supset z$.
\end{proof}

Let $(\alga,\pp,\{\tr,\fa\})\in\mathfrak{N}_{w}^{1}$. For any $X\subseteq A$, we set \[X^{\pp}:=\{y\in A: x\pp y,\text{ for any }x\in X\}.\]
It is well known that, given a set $A\ne\emptyset$ and a symmetric and irreflexive binary relation $\pp\ \subseteq A^{2}$, the mapping ${}^{\pp\pp}:\mathcal{P}(A)\to\mathcal{P}(A)$ such that, for any $X\subseteq A$, $X\mapsto X^{\pp\pp}$ is a closure operator (see e.g. \cite[Theorem 4]{Fazio1}). Moreover, the set ${C}(\mathcal{M})=\{X\subseteq A: X^{\pp\pp}=X\}$ forms a complete ortholattice $\mathscr{O}^{\mathcal{M}}=({C}(\mathcal{M}),\land,\lor,{}^{\ast},\emptyset,A)$ with $\emptyset$ and $A$ as its bottom and top element, respectively, and operations $$X\land Y=X\cap Y,\quad X\lor Y=(X\cup Y)^{\pp\pp},\text{ and }X^{*}=X^{\pp}.$$ We call $X\in{C}(\mathcal{M})$ \emph{closed}. In particular, note that $X^{\pp\pp\pp}=X^{\pp}$, $\emptyset^{\pp}=A$ and $A^{\pp}=\emptyset$. See e.g. \cite{Fazio1} for details.

\begin{remark}\label{rem:asd}For any $\mathfrak{N}_{w}^{1}$-model $\mathcal{M}=(\alga,\pp,\{\tr,\fa\})$, $\{x\}^{\pp\pp}=\{x^{*}\}^{\pp}$. To see this, note that $x\pp x^{*}$ entails that $x\in \{x^{*}\}^{\pp}$, and so $\{x\}^{\pp\pp}\subseteq\{x^{*}\}^{\pp}$ and $\{x^{*}\}^{\pp\pp}\subseteq\{x\}^{\pp}$. Conversely, let $y\in \{x\}^{\pp}$ and $z\in \{x^{*}\}^{\pp}$. One has $y\pp x$, $x^{*}\pp z$ and so, by $\mathrm{(i)}$, $y\pp z$. Since $z$ is arbitrary, $y\in\{x^{*}\}^{\pp\pp}$.  
\end{remark}
An obvious observation is that, for any $\mathcal{M}=(\alga,\pp,\{\fa,\tr\})\in\mathfrak{N}_{w}$, $\mathcal{I}(\alga)$ and $\mathscr{O}^{\mathcal{M}}$ can be regarded as first-order structures in a signature containing a binary relation ($\precsim$ and $\subseteq$, respectively), and a unary function (${}^{*}$ and ${}^{\pp}$, respectively).
\begin{proposition}\label{prop:ortholat}Let $\mathcal{M}=(\alga,\pp,\{\fa,\tr\})\in\mathfrak{N}_{w}$. Then $\mathcal{M}\in\mathfrak{N}_{w}^{1}$ if and only if the mapping $h:A\to{C}(\mathcal{M})$ such that $h:x\mapsto\{x\}^{\pp\pp}$ induces a first-order embedding between $\mathcal{I}(\alga)$ and $\mathscr{O}^{\mathcal{M}}$, once regarded as first-order structures in the same signature.
\end{proposition}
\begin{proof}
Concerning the right-to-left direction, observe that, since $\mathcal{I}(\alga)$ is embeddable  into $\mathscr{O}^{\mathcal{M}}$ via the map $h$, one has that $x\pp y^{*}$ iff $\{x\}^{\pp\pp}\subseteq\{y\}^{\pp\pp}$. Consequently, $x\pp y^{*}$ and $y\pp z^{*}$ imply $\{x\}^{\pp\pp}\subseteq\{y\}^{\pp\pp}$ and $\{y\}^{\pp\pp}\subseteq\{z\}^{\pp\pp}$ and so $\{x\}^{\pp\pp}\subseteq\{z\}^{\pp\pp}$ which, in turn, entails that $x\pp z^{*}$. As regards the converse direction, one has $x\precsim y$ implies $x\in\{y^{*}\}^{\pp}$. So $\{x\}^{\pp\pp}\subseteq\{y\}^{\pp\pp}$ (Remark \ref{rem:asd}). The converse follows upon noticing that $x\in\{x\}^{\pp\pp}$ and $y^{*}\in\{y\}^{\pp}$. That $h(x^{*})=h(x)^{\pp}$ follows again by Remark \ref{rem:asd}. 
\end{proof}
\noindent By Proposition \ref{prop:ortholat}, for any $\mathfrak{N}_{w}^{1}$-model $(\alga,\pp,\{\tr,\fa\})$, $\mathcal{I}(\alga)=(A,\precsim,{}^{*})$ can be orthoembedded into the complete ortholattice $C(\mathcal{M})$. Now, it can be seen that $\mathscr{O}^{\mathcal{M}}$ is the Dedekind-MacNeille completion (DM-completion) of $\mathcal{I}(\alga)$. Indeed, it is orthoisomorphic to the complete (bounded) lattice with antitone involution $\mathbb{DM}(\mathcal{I}(\alga))=(\mathrm{DM}(\mathcal{I}(\alga)),\land,\lor,{}^{*},\emptyset,A)$, where $\mathrm{DM}(\mathcal{I}(\alga))=\{X\subseteq A:L(U(X))=X\}$ and, for any $X,Y\in \mathrm{DM}(\mathcal{I}(\alga))$, $X\land Y = X\cap Y$, $X\lor Y = L(U(X\cup Y))$, and $X^{*}= L(X')$, where $X'=\{x^{*}:x\in X\}$. This construction turns out to be particularly interesting since any poset with antitone involution $\algp$  is orthoembeddable into its Dedekind-MacNeille completion $\mathbb{DM}(\algp)$ in such a way that existing joins and meets are always preserved. Moreover, any element of $\mathbb{DM}(\algp)$ is the join and the meet of (images of) elements of $\algp$, namely, for any $X\in \mathrm{DM}(\algp)$: 
\[X={\bigvee_{c\in C}}^{\mathbb{DM}(\algp)}c={\bigwedge_{d\in D}}^{\mathbb{DM}(\algp)}d,\] for some $C,D\subseteq P$. Note that, with a slight notational abuse, we identify elements of $\algp$ with their image in $\mathbb{DM}(\algp)$. 
$\mathbb{DM}(\algp)$ is, up to isomorphism, the least join and meet completion of $\algp$. See \cite{Fazio1} for details.\\ 
It is well known that any poset with antitone involution (orthoposet) admits a DM-completion which is always a complete bounded lattice with antitone involution (ortholattice). However, over the past years, much effort has been devoted to characterize posets with antitone involution, in particular orthoposets, whose DM-completion is not only an ortholattice but also an orthomodular lattice, the latter being a concept of prominent importance for quantum logic, the foundation of Quantum Mechanics, and, in general, for algebraic investigations of logic. See \cite{Beran,DallaChiara} for a thorough treatment of the topic.
\begin{definition}An ortholattice  $\alga=(A,\land,\lor,{}^{*},0,1)$ is said to be \emph{orthomodular} provided that  it satisfies the condition \[x\leq y\text{ and }x^{*}\land y = 0\text{ imply }x = y.\]
 
\end{definition}
Now, it naturally raises the question if there is at least an $\mathfrak{N}_{w}^{1}$-model $(\alga,\pp,\{\tr,\fa\})$ such that $C(\mathcal{M})$ is not only an otholattice, but also an \emph{orthomodular lattice}. Unfortunately, the answer is negative due to the next proposition.
\begin{proposition}\label{prop:d11}For any $(\alga,\pp,\{\tr,\fa\})\in\mathfrak{N}_{w}^{1}$, $\mathcal{I}(\alga)$ is such that, for any $x,y\in A$:
\[x\precsim y\text{ implies that, for any }z\in A,\ z\not\precsim x^{*}\text{ or }z\not\precsim y.\]
\end{proposition}
\begin{proof}Let $x\precsim y$.  Suppose by way of contradiction that there is $z\in A$ such that $z\precsim x^{*}$ and $z\precsim y$. One has that $y^{*}\precsim x^{*}$, so $y^{*}\oplus z\leq x^{*}\oplus z$. Therefore, $(y^{*}\oplus z)\otimes x\leq(x^{*}\oplus z)\otimes x\leq z$. However, one has also $(y^{*}\oplus z)\otimes x\precsim z^{*}$. Indeed, we have $(y^{*}\oplus z)\otimes x\precsim z^{*}$ iff $y^{*}\oplus z\precsim x^{*}\oplus z^{*}$ iff $x\otimes z\precsim y\otimes z^{*}$. One has $x\precsim z^{*}$ and $z\precsim y$ implies $x\otimes z\precsim x\otimes y\precsim z^{*}\otimes y$. We conclude that $(y^{*}\oplus z)\otimes x\precsim z,z^{*}$ which is impossible.
\end{proof}
\begin{corollary}For any non-trivial $(\alga,\pp,\{\tr,\fa\})\in\mathfrak{N}_{w}^{1}$, $\mathbb{DM}(\mathcal{I}(\alga))=(\mathrm{DM}(\mathcal{I}(\alga)),\land,\lor,{}^{*},\emptyset,A)$ is \emph{not} orthomodular.
\end{corollary}
\begin{proof} Observe that, by Proposition \ref{prop:df}, there are $a,b\in A$ such that $a\ra b \precsim a\circ b$ but $a\circ b\not\precsim a\ra b$. Therefore, $\{a\ra b\}^{\pp\pp}\subsetneq\{a\circ b\}^{\pp\pp}$. However, by Proposition \ref{prop:d11}, $\{(a\ra b)\}^{\pp}\cap\{a\circ b\}^{\pp\pp}=\emptyset$. 
\end{proof}
In general, we have the following result. Given a partially ordered set $(P,\leq)$, $x,y\in P$ are said to be \emph{incomparable}, written $x\parallel y$, provided that $x\not\leq y$ and $y\not\leq x$. Moreover, $x,y$ are said to be \emph{comparable}, written $x\not\parallel y$, if they are not incomparable.
\begin{proposition}\label{prop:topogigio}For any $(\alga,\pp,\{\tr,\fa\})\in\mathfrak{N}_{w}^{1}$, $\mathcal{I}(\alga)$ is such that, for any $x,y,z\in A$:
\[x\not\parallel z^{*}\text{ and }y\not\parallel z\text{ imply }L(x,y)=\emptyset=U(x,y).\]
\end{proposition}
\begin{proof}
First, let us observe that $x\not\parallel z^{*}\text{ and }y\not\parallel z$ entails that $x\not\precsim y$ and $y\not\precsim x$. Indeed, let us assume that $x\precsim z^{*}$. If $y\precsim x$, then $y\precsim z^{*}$. However, this entails that $y\not\precsim z$ (otherwise $y\precsim z,z^{*}$ which is impossible) and $z\not\precsim y$ as well otherwise one would have $z\precsim z^{*}$. If $x\precsim y$ then $z\precsim y$ is excluded by Proposition \ref{prop:d11}, otherwise one would have $z\precsim x^{*}$ and $z\precsim y$. The case  $y\precsim z$ is impossible as well, since, otherwise we would have $x\precsim z,z^{*}$. As regards the case $z^{*}\precsim x$, observe that if $x\precsim y$ one has that $y\not\precsim z$. So $z\precsim y$. However this entails that $y^{*}\precsim z^{*}\precsim x\precsim y$ which is again impossible. If $y\precsim x$, then clearly $z\not\precsim y$. However, this means that $y\precsim z$ which is excluded again by Proposition \ref{prop:d11}, since we would have $x^{*}\precsim z$, $y\precsim x^{**}$ and $y\precsim z$. 
Therefore, we can suppose that $x\parallel y$. Let us consider cases depending on that (a) $x\precsim z^{*}$ or (b) $z^{*}\precsim x$.\\
(a) If $x\precsim z^{*}$, we observe that the case $y\precsim z$  is trivial since, otherwise, one would have a $c\precsim z,z^{*}$ which is impossible. So we confine ourselves to consider the case $z\precsim y$. In this case, if $c\precsim x, y$, then one would have that $c\precsim z^{*}$ and $c\precsim y$ which is excluded by Proposition \ref{prop:df}.\\ 
(b) If $z^{*}\precsim x$, then if $z\precsim y$, one would have $y^{*}\precsim x$ and so the existence of $c\precsim y,x$ is excluded by Proposition \ref{prop:df}. If $y\precsim z$, then  $c\precsim x,y$ implies $c\precsim z$ and $c\precsim x$ which is again impossible.\\
Now, observe that $x\not\parallel z$ and $y\not\parallel z^{*}$ iff $x^{*}\not\parallel z^{*}$ and $y^{*}\not\parallel z$. From the first part of the proof it follows that $L(x^{*},y^{*})=\emptyset$ and so $U(x,y)=\emptyset$.
\end{proof}
In other words, for any $\mathfrak{N}_{w}^{1}$-model, any pair of propositions which are comparable w.r.t. contradictory propositions cannot contain, or be contained in a common proposition. 

We observe that $\mathfrak{N}_{w}^{1}$-models yield further logics that, in our opinion, are worth of attention. Let us define the relation $\vdash_{\pp}\ \subseteq\mathcal{P}(\mathrm{Fm}_{\mathcal{L}})\times\mathrm{Fm}_{\mathcal{L}}$ such that, for any $\Gamma\cup\{\varphi\}\subseteq\mathrm{Fm}_{\mathcal{L}}$, $\Gamma\vdash_{\pp}\varphi$ if and only if there is $\Gamma'=\{\gamma_{1},\dots,\gamma_{n}\}\subseteq\Gamma$ ($n\geq 1$) such that, for any $(\alga,\pp,\{\tr,\fa\})\in\mathfrak{N}_{w}^{1}$, and $h\in\mathbf{Hom}(\mathbf{Fm}_{\mathcal{L}},\alga)$
\[\{h(\gamma_{1})^{*},\dots,h(\gamma_{n})^{*}\}^{\pp}\subseteq\{h(\varphi)^{*}\}^{\pp},\] and we set $\vdash_{\pp}\varphi$ if and only if $\{h(\varphi)^{*}\}^{\pp}=A$ (for any $h\in\mathbf{Hom}(\mathbf{Fm}_{\mathcal{L}},\alga)$).
It can be seen that $\vdash_{\pp}$ is a sentential logic that, once interpreted according to a Nelsonian perspective, it expresses the fact that a formula $\varphi$ is a consequence of a set $\Gamma$ of assumptions provided that, for some finite $\Gamma'\subseteq\Gamma$, and for any interpretation $h$ of formulas into a $\mathfrak{N}_{w}^{1}$-model, any proposition ``containing'' the ``meaning'' of formulas in $\Gamma'$ contains also the meaning of $\varphi$ under the same interpretation. Note that, for any $\varphi,\psi\in\mathrm{Fm}_{\mathcal{L}}$, one has
$\varphi\vdash_{\pp}\psi$ iff for any $\mathfrak{N}_{w}^{1}$-model $\mathcal{M}=(\alga,\pp,\{\tr,\fa\})$, and any $h\in\mathbf{Hom}(\alga,\mathbf{Fm}_{\mathcal{L}})$, one has $\{h(\varphi)^{*}\}^{\pp}\subseteq\{h(\psi)^{*}\}^{\pp}$ iff  $h(\varphi)\pp h(\psi)^{*}$ iff   $h(\varphi\ra\psi)\pp\fa$. This means that $\varphi\vdash_{\pp}\psi$ if and only if $\vdash_{\mathsf{NeL}^{1}}\varphi\ra\psi$. However we have also that $\vdash_{\nel^{1}}\ \not\subseteq\ \vdash_{\pp}$ since, otherwise, one would have that $\vdash_{\mathsf{NeL}^{1}}(\varphi\otimes\psi)\ra\psi$ which is impossible (recall Proposition \ref{prop:trivial}). Moreover, $\vdash_{\pp}$ is \emph{theoremless}, since in no non-trivial $\mathfrak{N}_{w}^{1}$-model there are elements which are incompatible with themselves. Actually, $\vdash_{\pp}$ is extremely weak. For example, $\vdash_{\pp}$ does not satisfy the inference rule $\varphi,\varphi\ra\psi\vdash\psi$. Indeed, upon considering the structure $\mathcal{M}=(\alga=(\{a,b,c,d,\},\otimes,\circ,{}^{*}),\pp,\{\tr,\fa\})$ such that $\alga$ is determined according to the Cayley tables
\begin{center}
\begin{tabular}{c| c c c c } 
$\circ$ & a & b & c & d \\ 
\hline  
a& b & a & a & b \\ 
b& a & b & b & c  \\
c& a & b & b & a  \\
d& b & c & a & b \\
\hline 
\end{tabular}\quad\quad
\begin{tabular}{c| c c c c } 
$\otimes$ & a & b & c & d\\
\hline  
a &  a & c & c & a\\ 
b&  c&  d& a& b\\
c&  c&  a& a& c\\
d&  a&  b& c& d\\
\hline 
\end{tabular}\quad\quad
\begin{tabular}{c| c } $\ast$ &  \\ 
\hline  
a &  b\\ 
b &  a\\
c &  d\\
d &  c\\
\hline 
\end{tabular}
\end{center}
and such that $\pp=\{(b,\fa),(d,\fa),(a,\tr),(c,\tr),(a,b),(a,c),(c,a),(b,a),(c,d),(d,c)\}$ one has that $a\pp a^{*}=b$, $a\pp (a\ra b)^{*}=a\circ a=b$ but $a\not\pp b^{*}=a$. This means that a suitable homomorphism $h$ such that $x\mapsto a$ and $y\mapsto b$ witnesses the failure of $x,x\ra y\vdash_{\pp}y$. Similarly, one can prove that $\varphi,\varphi\supset\psi\not\vdash_{\pp}\psi$. This can be seen upon considering the structure $\mathcal{M}=(\alga=(\{a,b,c,d,\},\otimes,\circ,{}^{*}),\pp,\{\tr,\fa\})$ this time such that $\alga$ is determined according to the Cayley tables
\begin{center}
\begin{tabular}{c| c c c c } 
$\circ$ & a & b & c & d \\ 
\hline  
a& c & d & c & d \\ 
b& d & c & b & c  \\
c& c & b & c & d  \\
d& d & c & d & c \\
\hline 
\end{tabular}\quad\quad
\begin{tabular}{c| c c c c } 
$\otimes$ & a & b & c & d\\
\hline  
a &  d & a & d & a\\ 
b&  a&  b& c& d\\
c&  d&  c& b& a\\
d&  a&  d& a& d\\
\hline 
\end{tabular}\quad\quad
\begin{tabular}{c| c } $\ast$ &  \\ 
\hline  
a &  b\\ 
b &  a\\
c &  d\\
d &  c\\
\hline 
\end{tabular}
\end{center}
and such that $\pp=\{(b,\fa),(d,\tr),(a,\tr),(c,\fa),(a,b),(a,d),(d,a),(b,a),(c,d),(d,c)\}$. One has $a\pp a^{*}=b$ and $a\pp (a^{*}\oplus b)^{*}= (b\oplus b)^{*}=c^{*}=d$ but, again, $a\not\pp b^{*}=a$.

A further sentential logic based on $\mathfrak{N}_{w}^{1}$-models can be defined as follows.  We set 
\[\Gamma\vdash_{\pp^{1}}\varphi\text{ iff either }\models_{\mathfrak{N}_{w}^{1}}\varphi\text{ or there are }\gamma_{1},\dots,\gamma_{n}\in\Gamma\text{ such that }\models_{\mathfrak{N}_{w}^{1}}(\gamma_{1}\otimes\dots\otimes\gamma_{n})\ra\varphi.\]
As customary, if $\models_{\mathfrak{N}_{w}^{1}}\varphi$, we write $\vdash_{\pp^{1}}\varphi$. Note that, although they have the same theorems, $\vdash_{\pp^{1}}\ \subsetneq\ \vdash_{\nel^{1}}$. Indeed, for any $\Gamma\cup\{\varphi\}\subseteq\mathrm{Fm}_{\mathcal{L}}$, one has that, if there are $\gamma_{1},\dots,\gamma_{n}\in\Gamma$ such that $\models_{\mathfrak{N}_{w}^{1}}\gamma_{1}\otimes\dots\otimes\gamma_{n}\ra\varphi$, then, for any $\mathfrak{N}_{w}^{1}$-model $(\alga,\pp,\{\tr,\fa\})$ and any  evaluation $h\in\mathbf{Hom}(\mathbf{Fm}_{\mathcal{L}},\alga)$,  $h(\gamma_{i})\pp\fa$ ($1\leq i\leq n$) implies $h(\gamma_{1}\otimes\dots\otimes\gamma_{n})\pp\fa$ and so $h(\varphi)\pp\fa$. However, it can be seen that e.g. the inference rule $\varphi,\varphi\ra\psi\vdash\psi$ is not a valid inference rule of $\vdash_{\pp^{1}}$. Indeed, upon considering the structure $\mathcal{M}=(\alga=(\{a,b,c,d,e,f\},\otimes,\circ,{}^{*}),\pp,\{\tr,\fa\})$ from Example \ref{ex.ceci}, one has that 
$(a\otimes (a\ra a))\ra a=c\ra a = (c\circ a^{*})^{*}=(c\circ b)^{*}=a^{*}=b\not\pp\fa$. Nevertheless, $\varphi,\varphi\supset\psi\vdash_{\pp^{1}}\psi$ holds. See Proposition \ref{lem:partordsemgroup}.

We leave to the reader the burden of deciding whether logics outlined in this last part of the paper deserve further investigation. We confine ourselves to point out that, in spite of its (all in all)  simplicity, the framework we deal with is rich enough to provide consequence relations based on Nelson's concepts even different from Nelsonian logics outlined so far.

We close this section with a graphical summary depicted in Figure \ref{fig:posetextensions} of extensions of $\nl$ encountered in this paper. Lines denote inclusion.
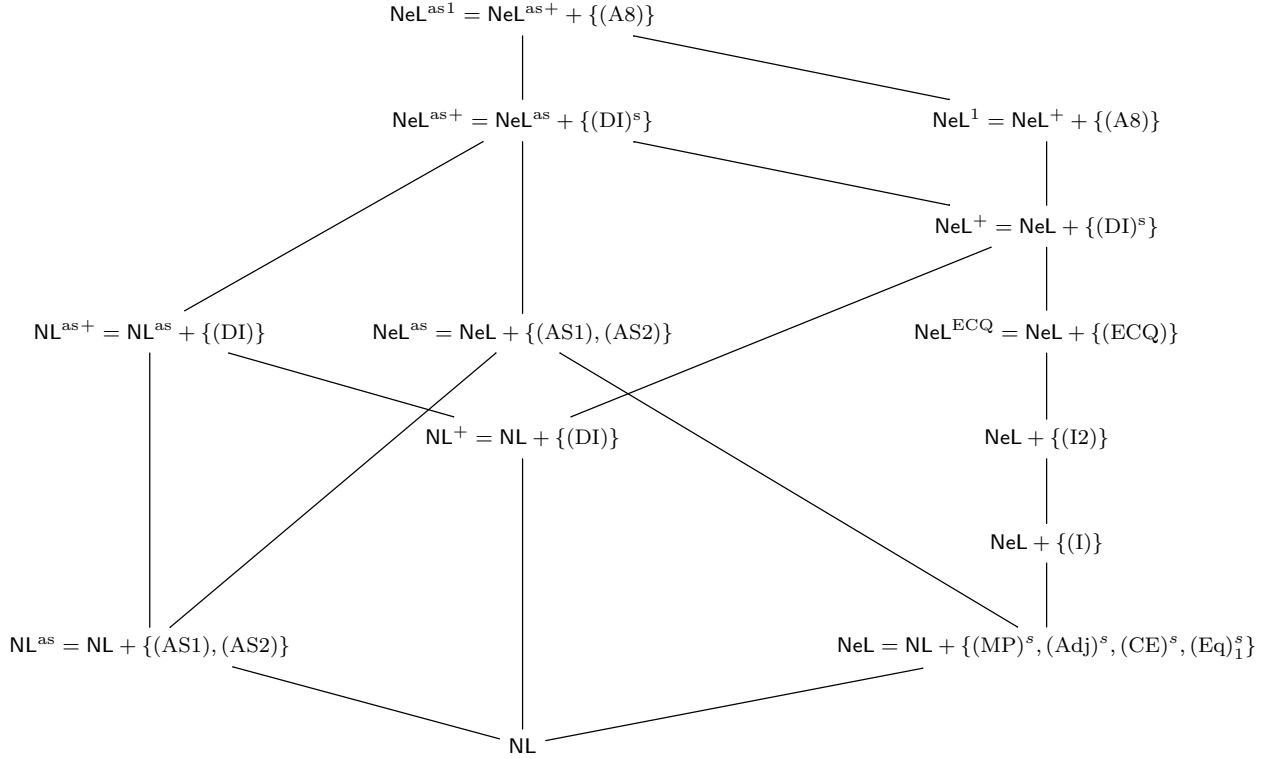
\begin{figure}[!htbp]
\[
\tiny{
\xymatrix{
&\nelas^{1}=\nelas^{+}+\{\mathrm{(A8)}\}\ar@{-}[drr]\ar@{-}[d]&&\\
&\nelas^{+}=\nelas+\{\mathrm{(DI)^{s}}\}\ar@{-}[dd]\ar@{-}[ddl]\ar@{-}[drr]&&\nel^{1}=\nel^{+}+\{\mathrm{(A8)}\}\ar@{-}[d]\\
&&&\nel^{+}=\nel+\{\mathrm{(DI)^{s}}\}\ar@{-}[d]\ar@{-}[ddll]\\
\nlas^{+}=\nlas+\{\mathrm{(DI)}\}\ar@{-}[ddd]\ar@{-}[dr] &\nelas=\nel+\{\eqref{as1},\eqref{as2}\}\ar@{-}[dddl]\ar@{-}[dddrr]&&\nel^{\mathrm{ECQ}}=\nel+\{\eqref{encq}\}\ar@{-}[d]\\
&\nl^{+}=\nl+\{\mathrm{(DI)}\}\ar@{-}[ddd]&&\nel+\{\eqref{i2}\}\ar@{-}[d]\\
&&&\nel+\{\eqref{i}\}\ar@{-}[d]\\
\nlas=\nl+\{\eqref{as1},\eqref{as2}\}\ar@{-}[dr]&&&\nel=\nl+\{\mathrm{(MP)}^{s},\mathrm{(Adj)}^{s},\mathrm{(CE)}^{s},\mathrm{(Eq)}_{1}^{s}\}\ar@{-}[dll]\\
&\nl&&
}}
\]
\caption{A summary of some extensions of $\nl$ ordered under set inclusion.}
\label{fig:posetextensions}
\end{figure}
\section{Conclusion}\label{sec:concl}
This work has been devoted to a preliminary investigation of logics arising from Everett John Nelson's early works. Our first aim has been providing a semantics for the logic $\nl$ outlined Nelson's \cite{Nelson3,Nelson1}, and for some of its prominent extensions, namely $\nlas,\nel,\nelas$. It has been shown that these systems are sound and complete with respect to (sub)classes of structures consisting of an algebraic and a relational part: $\mathfrak{N}_{w}$-models. As an output, it turns out that Nelson's logics admit a quite natural notion of truth.\\
Subsequently, we have focused on extensions of Nelsonian logics which, in our opinion, are able to grasp better than the previously introduced ones some of the key features of Nelson's thought. Specifically, we have investigated extensions of Nelsonian logics by means of inference rules ensuring the irreflexivity of incompatibility relations their models are endowed with, thus enjoying properties of prominent importance for connexive logic. Think e.g. to strong connexivity. Among logics investigated in this paper, we have introduced $\nl^{+}$ and $\nel^{+}$ ($\nlas^{+},\nelas^{+}$) whose motivation is strongly rooted in Nelson's arguments.\\
Finally, we dealt with the logic $\nel^{1}$ (and its extensions), namely the logic obtained upon extending $\nel^{+}$ with an inference rule encoding a weak form of transitivity of entailment. It turns out that the semantics of $\nel^{1}$ allows us to establish somewhat surprising relationships between algebraic-relational models of some Nelsonian logics and the order theory of structures, in particular ortholattices, which have been deeply investigated in a seemingly unrelated context: quantum logic. 

$\nl$ is the first example of a full-fledged connexive system in the whole history of connexive logic. Its extensions have an intuitive (although not unproblematic) motivation, and a (all in all) transparent semantics. Therefore, beside providing solutions to problems already scattered along the paper, we believe that a further investigation into properties of systems discussed so far might be worth of being pushed forward in several research directions.
\begin{itemize}
 \item In our opinion, it would be interesting to investigate further extensions of systems we dealt with in this paper with axioms codifying Nelson's ideas about entailment that we have not analyzed in this work. For example, as Nelson states in \cite[p. 74]{Nelson3}, while conjunction assumed by \emph{Principia Mathematica} presupposes an ``extrinsic'' relation between truth values (a conjunction is true if and only if each of its conjunct is true) intensional conjunction presupposes that ``If $pq$ holds, then $p$ and $q$ are consistent; and if $p$ and $q$ are not consistent, then $pq$ fails''. This fact can be rendered formally by the following axiom scheme:
\[(\varphi\otimes\psi)\ra(\varphi\circ\psi)\tag{A9}\label{a9}\] which, in view of axioms and inference rules  of $\nl^{+}$, is equivalent to \[(\varphi\otimes(\varphi\ra\psi))\ra\psi.\tag{$\mathrm{A9}^{*}$}\label{a9ast}\] Given the remarks in the previous section, \eqref{a9} is not provable in any of the systems outlined in this paper.  From a Nelsonian perspective, this could be seen as desirable, since assuming \eqref{a9} and so \eqref{a9ast} would conflict with Nelson's view on conjunction (think e.g. to the formula obtained from the replacement of $\psi$ by $\varphi$ in \eqref{a9ast}).  However, we believe that logics such as $\mathsf{L}+\{\eqref{a9}\}$ ($\mathsf{L}\in\{\nl^{+},\nlas^{+},\nel^{+},\nelas^{+}\}$) could constitute interesting subjects for future research.
\item {A further and perhaps one of the most important issues to be addressed in the future is investigating the logical system resulting from modifying $\nl$ and its extensions according to  Nelson's considerations in \cite{Nelson2}, namely by admitting \eqref{simpl} and replacing (A6) with ($\mathrm{A6}^{*}$).} 
\item On the semantical side, beside providing a solution to Problem \ref{prob:algebraiz}, it would be of interest investigating the order-theory of bounded posets with antitone involution associated to $\mathfrak{N}_{w}^{1}$-models. For example, it might be  worthwhile to characterize posets with antitone involution which can be obtained from $\mathfrak{N}_{w}^{1}$-models by means of incompatibility relations.
\item As regards the most algebra-oriented aspects of our research, a simple but meaningful observation grants that any $\mathfrak{N}_{w}^{1}$-model yields a (non-lattice ordered) partially ordered commutative involutive residuated groupoid, see Proposition \ref{lem:partordsemgroup} and Proposition \ref{prop:topogigio}. Therefore, it naturally raises the question whether it is possible to characterize structures of this kind arising in the framework of Nelson's logics.
\item Finally, from the genuinely proof-theoretical perspective, it might be of interest providing an in-depth investigation of sequent calculi or tableaux systems for logics we dealt with in this paper with an eye to establishing their decidability.
\end{itemize}
\quad\\
\begin{large}
\textbf{Declarations} 
\end{large}
\quad\\
\quad\\
\textbf{Funding}:
This work has been funded by the European Union - NextGenerationEU, Mission 4, Component 1, under the Italian Ministry of University and Research (MUR) National Innovation Ecosystem grant ECS00000041 - VITALITY -  CUP: C43C22000380007.
\quad\\
\quad\\
\textbf{Acknowledgements}: The first author gratefully thanks Francesco Paoli for having introduced him to E. J. Nelson's logic,  and for the insightful discussions of both historical and more technical concern they had over the past years about the topics of the present work. Furthermore, both authors thank Pierluigi Graziani for his insightful comments concerning the subject of this paper, and two anonymous referees for comments that helped to improve the shape and content of the present work considerably.

\end{document}